\documentclass{amsart}
\usepackage[margin=1.00in]{geometry}

\usepackage{amsfonts}

\usepackage{setspace}
\usepackage{graphicx}

\usepackage{hyperref}
\usepackage{graphicx}
\usepackage{amsmath, amsthm, amscd, amsfonts, amssymb, graphicx, color}
 \usepackage{url}
\usepackage{enumerate}
 \makeatletter
\let\reftagform@=\tagform@
\def\tagform@#1{\maketag@@@{(\ignorespaces\textcolor{blue}{#1}\unskip\@@italiccorr)}}
\renewcommand{\eqref}[1]{\textup{\reftagform@{\ref{#1}}}}
\makeatother
\usepackage{hyperref}
\hypersetup{colorlinks=true, linkcolor=red, anchorcolor=green,
citecolor=cyan, urlcolor=red, filecolor=magenta, pdftoolbar=true}

\usepackage{epsfig}        
\usepackage{epic,eepic}       

\newtheorem{theorem}{Theorem}
\theoremstyle{plain}

\newtheorem{corollary}{Corollary}

\newtheorem{definition}{Definition}
\newtheorem{example}{Example}

\newtheorem{remark}{Remark}

\numberwithin{equation}{section}

  \def\etal{et al.\,}

\begin{document}
\title[Popoviciu's type inequalities]{Popoviciu's type inequalities for $h$-${\rm{MN}}$-convex functions}
\author[M.W. Alomari]{M.W. Alomari}

\address{Department of Mathematics, Faculty of Science and
	Information Technology, Irbid National University, 2600 Irbid
	21110, Jordan.} \email{mwomath@gmail.com}
\date{\today}
\subjclass[2000]{26D15}

\keywords{Popoviciu's inequality,  $h$-convex function, Means}

\begin{abstract}
	In this work, Popoviciu type inequalities for
	$h$-${\rm{MN}}$-convex functions are proved, where M and N are specific mathematical means. Some direct examples
	are pointed out.
\end{abstract}
\maketitle
\section{Introduction}

We recall that, a function ${\rm{M}}:(0,\infty) \to (0,\infty)$ is
called a Mean function if
\begin{enumerate}
	\item  {\rm{Symmetry:}} ${\rm{M}}\left(x,y\right)={\rm{M}}\left(y,x\right)$.
	
	\item {\rm{Reflexivity:}} ${\rm{M}}\left(x,x\right)=x$.
	
	\item {\rm{Monotonicity:}} $\min\{x,y\} \le {\rm{M}}\left(x,y\right) \le \max\{x,y\}$.
	
	\item {\rm{Homogeneity:}} $ {\rm{M}}\left(\lambda x,\lambda y\right)=\lambda
	{\rm{M}}\left(x,y\right)$, for any positive scalar $\lambda$.
\end{enumerate}

The most famous and old known mathematical means are listed as
follows:
\begin{enumerate}
	\item The arithmetic mean :
	$$A := A\left( {\alpha ,\beta } \right) = \frac{{\alpha  + \beta
	}}{2},\,\,\,\,\,\alpha ,\beta \in  \mathbb{R}_+.$$
	
	\item The geometric mean :
	$$G: = G\left( {\alpha ,\beta } \right) = \sqrt {\alpha \beta },\,\,\,\,\,\alpha ,\beta \in \mathbb{R}_+$$
	
	\item The harmonic mean :
	$$H: = H\left(
	{\alpha ,\beta } \right) = \frac{2}{{\frac{1}{\alpha } +
			\frac{1}{\beta }}},\,\,\,\,\,\alpha ,\beta \in \mathbb{R}_+
	- \left\{ 0 \right\}.$$
\end{enumerate}
In particular, we have the famous inequality $ H \le G \le A$.

In 2007, Anderson \etal in \cite{AVV} developed a systematic study
to the classical theory of continuous and midconvex functions, by
replacing a given mean instead of the arithmetic mean.
\begin{definition}
	\label{def4}Let $f : I \to \left(0,\infty\right)$ be a
	continuous function where $I \subseteq (0,\infty)$. Let ${\rm{M}}$
	and ${\rm{N}}$ be any two Mean functions. We say $f$ is
	${\rm{{\rm{MN}}}}$-convex (concave) if
	\begin{align}
	f \left({\rm{M}}\left( x, y\right)\right) \le  (\ge) \,
	{\rm{N}}\left(f (x), f (y)\right), \label{eq1.5}
	\end{align}
	for  all $x,y \in I$ and $t\in [0,1]$.
\end{definition}
In fact, the authors in \cite{AVV} discussed the midconvexity of
positive continuous real functions according to some Means. Hence,
the usual midconvexity is a special case when both mean values are
arithmetic means. Also, they studied the dependence of
${\rm{MN}}$-convexity on ${\rm{M}}$ and ${\rm{N}}$ and give
sufficient conditions for ${\rm{MN}}$-convexity of functions
defined by Maclaurin series.   For other works regarding
${\rm{MN}}$-convexity see \cite{N} and \cite{NP}.

The  class of $h$-convex functions, which generalizes convex,
$s$-convex (denoted by $K_s^2$, \cite{B1}--\cite{B3}, \cite{HL}),
Godunova-Levin functions (denoted by $Q(I)$, \cite{GL}) and
$P$-functions (denoted by $P(I)$, \cite{PR}), was introduced by
Varo\v{s}anec in \cite{V}. Namely, the $h$-convex function is
defined as a non-negative function $f : I \to \mathbb{R}$ which
satisfies
\begin{align*}
f\left( {t\alpha +\left(1-t\right)\beta} \right)\le
h\left(t\right) f\left( {\alpha} \right)+ h\left(1-t\right)
f\left( {\beta} \right),
\end{align*}
where $h$ is a non-negative function, $t\in ¸ (0, 1)\subseteq J$
and $x,y \in I $, where $I$ and $J$ are real intervals such that
$(0,1) \subseteq J $. Accordingly, some properties of $h$-convex
functions were discussed in the same work of Varo\v{s}anec.\\

Let $h: J\to \left(0,\infty\right)$ be a non-negative
function. Define the
function ${\rm{M}}:\left[0,1\right]\to \left[a,b\right]$ given by
${\rm{{\rm{M}}}}\left(t\right)={\rm{{\rm{M}}}}\left( {t;a,b}
\right)$; where by ${\rm{{\rm{M}}}}\left( {t;a,b} \right)$ we mean
one of the following functions:
\begin{enumerate}
	\item $A_h\left( {a,b} \right):=h\left( {1 - t} \right)a +  h\left( {t} \right)b
	;\qquad \text{The generalized Arithmetic
		Mean}$.\\
	
	\item $G_h\left( {a,b}
	\right)=a^{h\left( {1-t} \right)} b^{h\left( {t} \right)};\qquad\qquad\,\,\,\,\,\,\, \text{The generalized Geometric Mean}$.\\
	
	\item $H_h\left( {a,b} \right):=\frac{ab}{h\left( {t} \right) a + h\left( {1 - t}
		\right)b} = \frac{1}{A_h\left( {\frac{1}{a},\frac{1}{b}} \right)};
	\qquad\,\,\,\,\, \text{The generalized
		Harmonic Mean}$.\\
\end{enumerate}
Note that ${\rm{{\rm{M}}}}\left( {h\left( {0} \right);a,b} \right)=a$ and
${\rm{M}}\left( {h\left( {1} \right);a,b} \right)=b$. Clearly, for $h(t)=t$ with $t=\frac{1}{2}$,
the means $A_{\frac{1}{2}}$, $G_{\frac{1}{2}}$ and
$H_{\frac{1}{2}}$, respectively; represents the midpoint of the
$A_{t}$, $G_{t}$ and $H_{t}$, respectively; which was discussed in
\cite{AVV} in viewing of Definition \ref{def4}.

For $h(t)=t$, we note that the above means are related with celebrated
AM-GM-HM inequality
\begin{align*}
H_t\left( {a,b} \right)\le G_t\left( {a,b} \right) \le  A_t\left(
{a,b} \right),\qquad \forall \,\, t\in [0,1].
\end{align*}
Indeed, one can easily prove  more general form of the above inequality; that is if $h$ is positive increasing on $[0,1]$ then the generalized AM-GM-HM inequality
is given by
\begin{align}
\label{AM-GM-HM}H_h\left( {a,b} \right)\le G_h\left( {a,b} \right) \le  A_h\left(
{a,b} \right),\qquad \forall \,\, t\in [0,1]\,\,\,\,\text{and}\,\,\,\,  a,b>0.
\end{align}

The Definition \ref{def4} can be extended according to the defined
mean ${\rm{{\rm{M}}}}\left( {t;a,b} \right)$, as follows:
Let $f : I \to \left(0,\infty\right)$
be any function. Let ${\rm{M}}$ and ${\rm{N}}$ be any two Mean
functions. We say $f$ is ${\rm{{\rm{MN}}}}$-convex (concave) if
\begin{align*}
f \left({\rm{M}}\left(t;x, y\right)\right) \le  (\ge) \,
{\rm{N}}\left(t;f (x), f (y)\right),
\end{align*}
for  all $x,y \in I$ and $t\in [0,1]$.

More generally, the authors of this paper have introduced the class of ${\rm{M_tN_h}}$-convex
functions by generalizing the concept of ${\rm{M_tN_t}}$-convexity
and combining it with $h$-convexity \cite{A}.
\begin{definition}\cite{A}
	\label{def5} Let $h: J\to \left(0,\infty\right)$ be a non-negative
	function. Let $f : I\to \left(0,\infty\right)$ be any function.
	Let ${\rm{M}}:\left[0,1\right]\to \left[a,b\right]$ and
	${\rm{N}}:\left(0,\infty\right)\to \left(0,\infty\right)$ be any
	two Mean functions. We say $f$ is $h$-${\rm{{\rm{MN}}}}$-convex
	(-concave) or that $f$ belongs to the class
	$\overline{\mathcal{MN}}\left(h,I\right)$
	($\underline{\mathcal{MN}}\left(h,I\right)$) if
	\begin{align}
	f \left({\rm{M}}\left(t;x, y\right)\right) \le  (\ge) \,
	{\rm{N}}\left(h(t);f (x), f (y)\right),\label{eq1.3}
	\end{align}
	for  all $x,y \in I$ and $t\in [0,1]$.
\end{definition}
Clearly, if ${\rm{M}}\left(t;x, y\right)=A_t\left( {x,y}
\right)={\rm{N}}\left(t;x, y\right)$, then Definition \ref{def5}
reduces to the original concept of $h$-convexity. Also, if we
assume $f$ is continuous,  $h(t)=t$ and $t=\frac{1}{2}$ in
\eqref{eq2.4}, then the Definition \ref{def5} reduces to the
Definition \ref{def4}.

The cases of $h$-${\rm{{\rm{MN}}}}$-convexity are given with
respect to a certain mean, as follow:
\begin{enumerate}
	\item $f$ is  ${\rm{A_tG_h}}$-convex iff
	\begin{align}
	f\left( {t\alpha  + \left( {1 -
			t} \right)\beta } \right) \le \left[ {f\left( \alpha  \right)}
	\right]^{h\left( t \right)} \left[ {f\left( \beta \right)}
	\right]^{h\left(1- t \right)}, \qquad  0\le t\le 1,\label{eqhAG}
	\end{align}

	\item $f$ is  ${\rm{A_tH_h}}$-convex iff
	\begin{align}
	f\left( {t\alpha  + \left( {1 -
			t} \right)\beta } \right) \le\frac{{f\left( \alpha \right)f\left(
			\beta  \right)}}{{h\left( 1-t \right)f\left( \alpha  \right) +
			h\left( { t} \right)f\left( \beta  \right)}}, \qquad  0\le t\le
	1.\label{eqhAH}
	\end{align}
	
	\item $f$ is  ${\rm{G_tA_h}}$-convex iff
	\begin{align}
	f\left( {\alpha ^t \beta ^{1 - t} } \right) \le h\left( {t}
	\right)f\left( \alpha \right) + h\left( {1 - t} \right)f\left(
	\beta  \right), \qquad 0\le t\le 1.\label{eqhGA}
	\end{align}
	
	\item $f$ is ${\rm{G_tG_h}}$-convex iff
	\begin{align}
	f\left( {\alpha ^t \beta ^{1 - t} } \right) \le \left[ {f\left(
		\alpha  \right)} \right]^{h\left( {t} \right)} \left[ {f\left(
		\beta  \right)} \right]^{h\left( {1 - t} \right)}, \qquad  0\le
	t\le 1.\label{eqhGG}
	\end{align}
	
	\item $f$ is ${\rm{G_tH_h}}$-convex iff
	\begin{align}
	f\left( {\alpha ^t \beta ^{1 - t} } \right) \le \frac{{f\left(
			\alpha  \right)f\left( \beta  \right)}}{{h\left( {1- t}
			\right)f\left( \alpha  \right) + h\left( { t} \right)f\left( \beta
			\right)}}, \qquad  0\le t\le 1.\label{eqhGH}
	\end{align}
	
	\item $f$ is  ${\rm{H_tA_h}}$-convex iff
	\begin{align}
	f\left( {\frac{{\alpha \beta }}{{t\alpha  + \left( {1 - t}
				\right)\beta }}} \right) \le h\left( {1-t} \right)f\left( \alpha
	\right) + h\left( { t} \right)f\left( \beta  \right), \qquad 0\le
	t\le 1.\label{eqhHA}
	\end{align}
	\item $f$ is  ${\rm{H_tG_h}}$-convex iff
	\begin{align}
	f\left( {\frac{{\alpha \beta }}{{t\alpha  + \left( {1 - t}
				\right)\beta }}} \right) \le \left[ {f\left( \alpha  \right)}
	\right]^{h\left( { 1-t} \right)} \left[ {f\left( \beta  \right)}
	\right]^{h\left( {  t} \right)}, \qquad 0\le t\le 1.\label{eqhHG}
	\end{align}
	
	\item $f$ is ${\rm{H_tH_h}}$-convex iff
	\begin{align}
	f\left( {\frac{{\alpha \beta }}{{t\alpha  + \left( {1 - t}
				\right)\beta }}} \right) \le \frac{{f\left( \alpha  \right)f\left(
			\beta  \right)}}{{h\left( {t} \right)f\left( \alpha  \right) +
			h\left( {1 - t} \right)f\left( \beta  \right)}}, \qquad  0\le t\le
	1.\label{eqhHH}
	\end{align}
\end{enumerate}

\begin{remark}
	In all previous cases,  $h(t)$ and $h(1-t)$ are not equal to zero
	at the same time. Therefore, if
	$h(0)=0$ and $h(1)=1$, then the Mean function ${\rm{N}}$ satisfying
	the conditions ${\rm{N}}\left( {h\left( 0 \right),f\left( x \right),f\left( y \right)}
	\right) = f\left( x \right)$ and $ {\rm{N}}\left( {h\left( 1
		\right),f\left( x \right),f\left( y \right)} \right) = f\left(y
	\right) $.
\end{remark}
\begin{remark}
	According to the Definition \ref{def5},  we may extend the classes
	$Q(I), P(I)$
	and $K_s^2$ by replacing the  arithmetic mean by another given
	one.
	Let ${\rm{M}}:\left[0,1\right]\to \left[a,b\right]$
	and ${\rm{N}}:\left(0,\infty\right)\to \left(0,\infty\right)$ be
	any two Mean functions.
	\begin{enumerate}
		\item Let $s\in (0,1]$, a function $f : I\to
		\left(0,\infty\right)$ is ${\rm{M_tN_{t^s}}}$-convex function or
		that $f$ belongs to the class
		$K^2_s\left(I;{\rm{M_t}},{\rm{N_{t^s}}}\right)$ if for all $x,y \in I$
		and $t\in [0,1]$ we have
		\begin{align}
		f \left({\rm{M}}\left(t;x, y\right)\right) \le {\rm{N}}\left(t^s;f
		(x), f (y)\right).\label{eq1.12}
		\end{align}
		
		\item  We say that $f : I \to \left(0,\infty\right)$ is an
		extended Godunova-Levin function or that $f$ belongs to the class
		$Q\left(I;{\rm{M_t}},{\rm{N_{1/t}}}\right)$ if for all $x,y \in I$ and
		$t\in (0,1)$ we have
		\begin{align}
		f \left({\rm{M}}\left(t;x, y\right)\right) \le
		{\rm{N}}\left(\frac{1}{t};f (x), f (y)\right).\label{eq1.13}
		\end{align}
		
		\item We say that  $f : I\to \left(0,\infty\right)$ is
		$P$-${\rm{M_tN_{t=1}}}$-function or that $f$ belongs to the class
		$P\left(I;{\rm{M_t}},{\rm{N_1}}\right)$ if for all $x,y \in I$ and
		$t\in [0,1]$ we have
		\begin{align}
		f \left({\rm{M}}\left(t;x, y\right)\right) \le {\rm{N}}\left(1;f
		(x), f (y)\right).\label{eq1.14}
		\end{align}
		In \eqref{eq1.12}--\eqref{eq1.14}, setting ${\rm{M}}\left(t;x,
		y\right)= {\rm{A_t}}\left(x, y\right)={\rm{N}}\left(t;x,
		y\right)$, we then refer to the original definitions of these
		class of convexities.
	\end{enumerate}
\end{remark}

\begin{remark}
	\label{remark2}Let $h$ be a non-negative function such that
	$h\left(t\right) \ge t$ for $t\in \left(0,1\right)$. For instance
	$h_r\left(t\right) = t^r$, $t\in \left(0,1\right)$ has that
	property. In particular, for $r\le 1$, if $f$ is a non-negative
	${\rm{M_tN_t}}$-convex function on $I$, then for $x,y\in I$, $t
	\in (0,1)$ we have
	\begin{align*}
	f \left({\rm{M}}\left(t;x, y\right)\right) \le {\rm{N}}\left(t;f
	(x), f (y)\right) \le {\rm{N}}\left(t^r;f (x), f (y)\right)=
	{\rm{N}}\left(h\left(t\right);f (x), f (y)\right),
	\end{align*}
	for all $r\le 1$ and $t\in \left(0,1\right)$. So that $f$ is
	${\rm{M_tN_h}}$-convex. Similarly, if the function satisfies
	the property $h\left(t\right) \le t$ for $t \in \left(0,1\right)$,
	then $f$ is a non-negative  ${\rm{M_tN_h}}$-concave. In
	particular, for $r\ge 1$, the function $h_r(t)$ has that property
	for $t\in \left(0,1\right)$.
	So that if $f$ is a non-negative ${\rm{M_tN_t}}$-concave
	function on $I$, then for $x,y\in I$, $t \in (0,1)$ we have
	\begin{align*}
	f \left({\rm{M}}\left(t;x, y\right)\right) \ge {\rm{N}}\left(t;f
	(x), f (y)\right) \ge {\rm{N}}\left(t^r;f (x), f
	(y)\right)={\rm{N}}\left(h\left(t\right);f (x), f (y)\right),
	\end{align*}
	for all $r\ge 1$ and $t\in \left(0,1\right)$, which means that $f$
	is ${\rm{M_tN_h}}$-concave.
\end{remark}

As known, it is not easy to determine whether a given function is
convex or not. Because of that, Jensen in \cite{J} proved his
famous characterization of convex functions. Simply, for a
continuous functions $f$ defined on a real interval $I$, $f$ is
convex if and only if
\begin{align*}
f\left( {\frac{{x + y}}{2}} \right) \le \frac{{f\left( x \right) +
		f\left( y \right)}}{2},
\end{align*}
for all $x,y\in I$.

In 1965, another characterization was presented by Popoviciu
\cite{P}, where he proved that the following theorem.
\begin{theorem}
	Let $f:I\to \mathbb{R}$ be continuous. Then, $f$ is convex if and
	only if
	\begin{align}
	\frac{2}{3}\left[{ f\left( {\frac{{x + z}}{2}} \right)+f\left(
		{\frac{{y + z}}{2}} \right)+f\left( {\frac{{x + y}}{2}}
		\right)}\right] \le f\left( {\frac{{x + y + z}}{3}} \right)
	+\frac{f\left( x \right)+f\left( y \right)+f\left( z \right)}{3},
	\label{eq1.15}
	\end{align}
	for all $x,y,z\in I$, and the equality occurred by $f(x)=x$, $x\in
	I$.
\end{theorem}
The corresponding version of  Popoviciu inequality for
${\rm{G_tG_t}}$-convex (concave) function was presented by
\cite{N}, where he  proved that for   all $x,y,z\in I$ the
inequality
\begin{align}
f^2\left( {\sqrt {xz}} \right)f^2\left( {\sqrt {yz}}
\right)f^2\left( {\sqrt {xy}} \right) \le (\ge)  f^3\left(
{\sqrt[3]{xyz} } \right) f\left( x \right)f\left( y \right)f\left(
z \right),
\end{align}
holds.

One of the most applicable benefits of Popoviciu's inequality is
to maximize and/or minimize a given function (or certain real
quantities) with out using derivatives, so that such type of
inequalities plays an important role in Optimizations and
Approximations. Another serious usefulness is to generalize some
old famous inequalities, e.g., the Popoviciu's inequality can be
considered as an elegant generalization of Hlawka's inequality
using convexity as a simple tool of geometry. For any real numbers
$x,y,z$, the Hlawka's inequality reads:
\begin{align}
\label{Hlawkaineq}\left| x \right| + \left| y \right| + \left| z
\right| + \left| {x + y + z} \right| \ge \left| {x + z} \right| +
\left| {z + y} \right| + \left| {x + y} \right|.
\end{align}
D. Smiley \& M.  Smiley \cite{W} (see also \cite{SS}, p. 756),
interpreted Hlawka's inequality geometrically by saying that:
``the total length over all sums of pairs from three vectors is
not greater than the perimeter of the quadrilateral defined by the
three vectors.'' For recent comprehensive history regarding
Hlawka's inequality see \cite{F}. It's convenient to note that, a
normed linear space for which inequality \eqref{Hlawkaineq} holds
for all $x$, $y$, $z$ is called a Hlawka space or quadrilateral
space, see \cite{TTH} and \cite{TTMT} (also \cite{SS}). For
instance, each inner product space is a Hlawka space, \cite{MPF}.

The extended version of Popoviciu's inequality to
several variables was not possible without the help of Hlawka's
inequality, as it inspired  the authors of \cite{BNP} to develop a
higher dimensional analogue of Popoviciu's inequality based on his
characterization. Interesting generalizations and counterparts of
Popoviciu inequality with some ramified consequences can be
found in \cite{G} and \cite{VS}.\\

Therefore, as Popoviciu's inequality  one of the most popular
generalization of Hlawka's inequality, and due to its important
usefulness, in this work we establish the corresponding Popoviciu
type inequalities according to a given mean used instead of the
arithmetic mean. Namely, for $h$-${\rm{AN}}$-convex functions
several inequalities of Popoviciu type are proved. In this way, we
extend Hlawka's inequality based on the geometric structure used
under an $h$-${\rm{AN}}$-convex mappings.

\section{Popoviciu type inequalities for $h$-${\rm{AN}}$-convex
	functions}
After focus consideration we find that, there is neither
nonnegative $\frac{1}{t}$-${\rm{M_tA_t}}$-concave nor
$\frac{1}{t}$-${\rm{M_tH_t}}$--convex functions, where
$M_t=A_t,G_t,H_t$. The same observation holds for
$h\left(t\right)=t^k$, $k\le -1$, $t\in (0,1)$.

To see how this holds, suppose on the contrary that there is a
nonnegative function $f$ which is
${\rm{M_tA_{1/t}}}$-concave on $I$. Thus, for Means
${\rm{M_t}}$ and ${\rm{A_t}}$, the reverse inequality of
\eqref{eq1.3} holds for all $x,y \in I$ and $t\in (0,1)$.
\begin{align*}
f\left( {M \left( {t;x,y} \right)} \right)
\ge \frac{1}{1-t} f\left( x \right) + \frac{1}{t} f\left( y
\right).
\end{align*}
Since $ M_t \left( {x,x} \right)=x$, so by setting $x=y$ we have
\begin{align*}
f\left( {x} \right) \ge \frac{1}{1-t} f\left( x \right) +
\frac{1}{t}f\left( x
\right) = \frac{1}{t\left(1-t\right)} f\left( x \right),
\end{align*}
which is equivalent to write $\left(t-t^2-1\right)f\left(x\right)
\ge 0$, $\forall t\in (0,1)$. But since $f$ is non-negative we
must have $t-t^2-1\ge0$,  $0<t<1$ which is impossible and thus we
got a contradiction. Hence, we must have $f\left(x\right) \le 0$.

In case when $f$ is nonnegative
${\rm{M_tH_{1/t}}}$--convex function, then
\begin{align*}
f\left( {M \left( {t;x,y} \right)} \right) \le   \frac{{t\left(
		1-t \right)f\left( x \right)f\left(y  \right)}}{{tf\left( x
		\right) +\left( 1-t \right)f\left( y  \right)}},
\end{align*}
setting $x=y$ we have
\begin{align*}
f\left( {x} \right)
\le   t\left(
1-t \right)f\left( x \right),
\end{align*}
and this is equivalent to write $\left(t\left( 1-t \right)
-1\right)f\left( x \right)\ge0$, since $f$ is nonnegative we must
have $t\left( 1-t \right) -1\ge0$ which impossible for $t\in
(0,1)$, which contradicts the nonnegativity assumption  of $f$.
Hence, $f\le0$.

\begin{remark}
	\label{cor}There is no nonnegative  ${\rm{M_tA_1}}$-concave nor
	${\rm{M_tH_{1}}}$-convex functions, where $M_t=A_t,G_t,H_t$. The
	proof is  simpler than that ones given above.
\end{remark}

According to the previous discussion, we need to extend the
classes $Q\left(I;{\rm{M_t}},{\rm{A_{1/t}}}\right)$,
$Q\left(I;{\rm{M_t}},{\rm{H_{1/t}}}\right)$,
$P\left(I;{\rm{M_t}},{\rm{A_1}}\right)$, and
$P\left(I;{\rm{M_t}},{\rm{H_1}}\right)$. Consequently, we say that
a function $f:I\to \mathbb{R}$
\begin{enumerate}
	\item is ${\rm{M_tA_{1/t}}}$-concave, if $-f \in
	Q\left(I;{\rm{M_t}},{\rm{A_{1/t}}}\right)$, i.e.,
	\begin{align*}
	f \left({\rm{M}}\left(t;x, y\right)\right)\ge \frac{1}{1-t}
	f\left( x \right) + \frac{1}{t} f\left( y
	\right),
	\end{align*}
	for all $x,y \in I$ and $t\in (0,1)$.\\

	\item is  ${\rm{M_tH_{1/t}}}$-convex, if $-f \in
	Q\left(I;{\rm{M_t}},{\rm{H_{1/t}}}\right)$, i.e.,
	\begin{align*}
	f \left({\rm{M}}\left(t;x, y\right)\right)\ge \frac{{t\left( 1-t
			\right)f\left( x \right)f\left(y  \right)}}{{tf\left( x \right)
			+\left( 1-t \right)f\left( y  \right)}},
	\end{align*}
	for all $x,y \in I$ and $t\in (0,1)$.\\
	
	\item is ${\rm{M_tA_1}}$-concave, if $-f \in
	P\left(I;{\rm{M_t}},{\rm{A_1}}\right)$, i.e.,
	\begin{align*}
	f \left({\rm{M}}\left(t;x, y\right)\right)\ge  f\left( x \right)
	+f\left( y
	\right)
	\end{align*}
	for all $x,y \in I$ and $t\in (0,1)$.\\
	
	\item is  ${\rm{M_tH_1}}$-concave, if $-f \in
	P\left(I;{\rm{M_t}},{\rm{H_1}}\right)$, i.e.,
	\begin{align*}
	f \left({\rm{M}}\left(t;x, y\right)\right)\ge\frac{{f\left( x
			\right)f\left( y \right)}}{{f\left( x \right) + f\left( y
			\right)}},
	\end{align*}
	for all $x,y \in I$ and $t\in (0,1)$.
\end{enumerate}

In the same way, there is no  ${\rm{M_tG_{1/t}}}$-concave
function satisfies $f\left(x\right)> 1$. To support this
assertion, assume there exists
${\rm{M_tG_{1/t}}}$-concave function, so that for Means
${\rm{M_t}}$ and ${\rm{G_t}}$, the reverse inequality of
\eqref{eq1.3} holds for all $x,y \in I$ and $t\in (0,1)$.
\begin{align*}
f\left( {M \left( {t;x,y} \right)} \right) \ge  \left[f\left( x
\right)\right]^{\frac{1}{1-t}}\left[ f\left( y
\right)\right]^ {\frac{1}{t}},
\end{align*}
since $ M_t \left( {x,x} \right)=x$, so by setting $x=y$ we have
\begin{align*}
f\left( {x} \right) \ge \left[f\left( x
\right)\right]^{\frac{1}{1-t}+ \frac{1}{t}},
\end{align*}
since $f\left(x\right)> 1$ and $t\in (0,1)$ then we must have
$\frac{1}{1-t}+ \frac{1}{t}\le1$ which is equivalent to write
$1\le t \left(1-t\right)$ for all $t\in (0,1)$ and this is
impossible, thus we have a contradiction. Hence, we must have $
0\le f\left(x\right) \le 1$.

\begin{remark}
	There is no $1$-${\rm{M_tG_t}}$-concave function satisfies
	$f\left(x\right)> 1$. The proof is  simpler than that ones given
	above.
\end{remark}

A function $h: I\to \mathbb{R}$ is said to be
\begin{enumerate}
	\item additive if $h\left( {s + t} \right) = h\left( s \right) +
	h\left( t \right)$,
	
	\item    subadditive if $h\left( {s + t} \right) \le h\left( s
	\right) + h\left( t \right)$,
	
	\item    superadditive if $h\left( {s + t} \right) \ge h\left( s
	\right) + h\left( t \right)$,
\end{enumerate}
for all $s,t \in I$. For example, let $h:I\to
\left(0,\infty\right)$ given by $h\left(x\right)=x^k$, $x>0$. Then
$h$ is
\begin{enumerate}
	\item additive if $k=1$.

	\item subadditive if $k \in \left(-\infty,-1\right]\cup
	\left[0,1\right)$.
	
	\item superadditive if $k \in \left(-1,0\right) \cup
	\left(1,\infty\right)$.
\end{enumerate}

We note here, in all next results  and for the classes
${\rm{M_tA_{1/t}}}$-concave,
${\rm{M_tG_{1/t}}}$-concave,
${\rm{M_tH_{1/t}}}$-convex ,
${\rm{M_tA_1}}$-concave, and   ${\rm{M_tH_1}}$-convex
functions,   $f$ is defined to be $f : I \to \mathbb{R}$,
$I\subseteq (0,\infty)$.

\subsection{The case when $f$ is $h$-${\rm{AA}}$-convex}

Now, we are ready to state our first main result.
\begin{theorem}
	\label{thm1} Let $h: I\to \left(0,\infty\right)$ be a non-negative
	super(sub)additive
	function. If $f : I \to \left(0,\infty\right)$ be an
	${\rm{A_tA_h}}$-convex (concave) function, then
	\begin{multline}
	f\left( {\frac{{x + z}}{2}} \right)+f\left( {\frac{{y + z}}{2}}
	\right)+f\left( {\frac{{x + y}}{2}} \right)
	\\
	\le\,(\ge)\, h\left( {3/2} \right) f\left( {\frac{{x + y + z}}{3}}
	\right) +h\left(1/2 \right)\left[ {f\left( x \right)+f\left( y
		\right)+f\left( z \right)} \right], \label{eq2.1}
	\end{multline}
	for all $x,y,z\in I$.
\end{theorem}

\begin{proof}
	$f$ is  ${\rm{A_tA_h}}$-convex iff the inequality
	\begin{align*}
	f\left( {t\alpha  + \left( {1 -
			t} \right)\beta } \right) \le  h\left( t \right) f\left( \alpha
	\right) +  h\left(1- t \right)f\left( \beta \right), \qquad 0\le
	t\le 1,
	\end{align*}
	holds for all $\alpha,\beta\in I$. Assume that $x\le y \le z$. If
	$y \le \frac{x+y+z}{3}$, then
	\begin{align*}
	\frac{{x + y + z}}{3} \le \frac{{x + z}}{2} \le z
	\,\,\text{and}\,\, \frac{{x + y + z}}{3} \le \frac{{y + z}}{2} \le
	z,
	\end{align*}
	so that there exist two numbers $s,t \in \left[0,1\right]$
	satisfying
	\begin{align*}
	\frac{{x + z}}{2} = s\left( {\frac{{x + y + z}}{3}} \right) +
	\left( {1 - s} \right)z,
	\end{align*}
	and
	\begin{align*}
	\frac{{y + z}}{2} = t\left( {\frac{{x + y + z}}{3}} \right) +
	\left( {1 - t} \right)z.
	\end{align*}
	Summing up, we get
	$\left(x+y-2z\right)\left(s+t-\frac{3}{2}\right)=0$. If
	$x+y-2z=0$, then $x = y = z$, and Popoviciu's inequality holds.
	
	If $s+t=\frac{3}{2}$, then since $f$ is ${\rm{A_tA_h}}$-convex, we
	have
	\begin{align*}
	f\left( {\frac{{x + z}}{2}} \right) = f\left[ {s\left( {\frac{{x + y + z}}{3}} \right) + \left( {1 - s} \right)z}
	\right]
	\le h\left( s \right) f\left( {\frac{{x + y + z}}{3}} \right)+h\left( {1 - s} \right) f\left( z
	\right),
	\end{align*}
	\begin{align*}
	f\left( {\frac{{y + z}}{2}} \right) = f\left[ {t\left( {\frac{{x + y + z}}{3}} \right) + \left( {1 - t} \right)z}
	\right]
	\le h\left( t \right) f\left( {\frac{{x + y + z}}{3}} \right)+h\left( {1 - t} \right) f\left( z
	\right),
	\end{align*}
	and
	\begin{align*}
	f\left( {\frac{{x + y}}{2}} \right) \le h\left( {1/2}
	\right)\left[ {f\left( x \right)+f\left( y \right)} \right].
	\end{align*}
	Summing up these inequalities taking into account that $h$ is
	superadditive we get
	\begin{align*}
	&f\left( {\frac{{x + z}}{2}} \right)+f\left( {\frac{{y + z}}{2}} \right)+f\left( {\frac{{x + y}}{2}} \right) \\
	&\le  h\left( s \right)f\left( {\frac{{x + y + z}}{3}} \right)+h\left( {1 - s} \right)f\left( z
	\right) +h\left( t \right)f\left( {\frac{{x + y + z}}{3}} \right)+h\left( {1 - t} \right)f\left( z \right)\\
	&\qquad+ h\left( {1/2}\right)\left[ {f\left( x \right)+f\left( y \right)}
	\right]\\
	&= \left[ {h\left( s \right) + h\left( t \right)} \right]f\left( {\frac{{x + y + z}}{3}} \right)+ \left[ {h\left( {1 - s} \right) + h\left( {1 - t} \right)}  \right]f\left( z \right)+  h\left( {1/2}\right)\left[ {f\left( x \right)+f\left( y \right)}
	\right] \\
	&\le h\left( {s + t} \right) f\left( {\frac{{x + y + z}}{3}} \right) +h\left( {2 - s - t} \right)  f\left( z \right)+h\left(1/2 \right)\left[ {f\left( x \right)+f\left( y \right)} \right]   \\
	&= h\left( {3/2} \right) f\left( {\frac{{x + y + z}}{3}} \right) +h\left( {1/2} \right)  f\left( z \right)+h\left(1/2 \right)\left[ {f\left( x \right)+f\left( y \right)} \right] \\
	&= h\left( {3/2} \right) f\left( {\frac{{x + y + z}}{3}} \right)    +h\left(1/2 \right)\left[ {f\left( x \right)+f\left( y \right)+f\left( z \right)}
	\right],
	\end{align*}
	as desired.
\end{proof}

\begin{remark}
Setting $z=y$ 	in \eqref{eq2.1}, then we have
	\begin{align*}
	2f\left( {\frac{{x + y}}{2}} \right)+f\left( {y}
	\right) \le \,(\ge)\,  h\left( {3/2} \right) f\left( {\frac{{x +
				2y}}{3}} \right) +h\left(1/2 \right)\left[ {f\left( x
		\right)+2f\left( y \right)} \right].
	\end{align*}
	for all $x,y\in I$.
\end{remark}

\begin{remark}
Setting $z=y$ 	in \eqref{eq2.1}, then we get
	\begin{align*}
	2f\left( {\frac{{x + y}}{2}} \right)+f\left( {y}
	\right) \le\,(\ge)\, h\left( {3/2} \right) f\left( {\frac{{x +
				2y}}{3}} \right) +h\left(1/2 \right)\left[ {f\left( x
		\right)+2f\left( y \right)} \right],
	\end{align*}
	for all $x,y\in I$.
\end{remark}

\begin{corollary}
\label{cor1}	Let $h: I\to \left(0,\infty\right)$ be a non-negative
	super(sub)additive function. If $f : I \to \left(0,\infty\right)$
	be an ${\rm{A_tA_t}}$-convex (concave) function, then
	\begin{align*}
	\frac{2}{3}\left[{ f\left( {\frac{{x + z}}{2}} \right)+f\left(
		{\frac{{y + z}}{2}} \right)+f\left( {\frac{{x + y}}{2}}
		\right)}\right] \le\,(\ge)\,  f\left( {\frac{{x + y + z}}{3}}
	\right) +\frac{f\left( x \right)+f\left( y \right)+f\left( z
		\right)}{3},
	\end{align*}
	for all $x,y,z\in I$. The equality holds when $f$ is affine.
\end{corollary}

\begin{example}
	\begin{enumerate}
		\item Let $f\left(x\right)=x^p$, $p\ge1$ then $f$ is
		${\rm{A_tA_t}}$-convex for all $x>0$. Applying Corollary
		\ref{cor1}, we get
		\begin{align*}
		\frac{2}{3}\left[{ \left( {\frac{{x + z}}{2}} \right)^p+\left(
			{\frac{{y + z}}{2}} \right)^p+ \left( {\frac{{x + y}}{2}}
			\right)^p}\right] \le \left( {\frac{{x + y + z}}{3}} \right)^p +
		\frac{x^p+y^p+z^p}{3},
		\end{align*}
		for all $x,y,z>0$.
		
		\item Let $f\left(x\right)=-\log x$, then $f$ is
		${\rm{A_tA_t}}$-convex for all $0<x<1$. Applying Corollary
		\ref{cor1}, we get
		\begin{align*}
		\left( {x + z} \right)^2 \left( {y + z} \right)^2 \left( {x + y}
		\right)^2  \ge \frac{{64}}{{27}}\left( {x + y + z} \right)^3
		\left( {xyz} \right) ,
		\end{align*}
		for all $1>x,y,z>0$.
	\end{enumerate}
\end{example}
\begin{corollary}
	\label{cor2} If $f : I \to \mathbb{R}$ be an
	${\rm{A_tA_{1/t}}}$-concave
	function, then
	\begin{align*}
	\frac{3}{2}\left[{ f\left( {\frac{{x + z}}{2}} \right)+f\left(
		{\frac{{y + z}}{2}} \right)+f\left( {\frac{{x + y}}{2}}
		\right)}\right]
	\le\,(\ge)\,  f\left( {\frac{{x + y + z}}{3}} \right)
	+3\left[{f\left( x \right)+f\left( y \right)+f\left( z
		\right)}\right],
	\end{align*}
	for all $x,y,z\in I$.
\end{corollary}

\begin{example}
	Let $f\left(x\right)= \log x$, then $f$ is an
	${\rm{A_tA_{1/t}}}$-concave for $0<x<1$. Applying
	Corollary \ref{cor2}, we get
	\begin{align*}
	\left( {x + z} \right)^3 \left( {y + z} \right)^3 \left( {x + y}
	\right)^3  \ge \frac{{512}}{9}\left( {x + y + z} \right)^2 \left(
	{xyz} \right)^6,
	\end{align*}
	for all $0< x,y,z<1$.
\end{example}

\begin{corollary}
	\label{cor3}
	If $f : I \to  \mathbb{R}$
	be an  ${\rm{A_tA_1}}$-concave function, then
	\begin{align*}
	f\left( {\frac{{x + z}}{2}} \right)+f\left( {\frac{{y + z}}{2}}
	\right)+f\left( {\frac{{x + y}}{2}} \right) \le\,(\ge)\,  f\left(
	{\frac{{x + y + z}}{3}} \right) + f\left( x \right)+f\left( y
	\right)+f\left( z \right),
	\end{align*}
	for all $x,y,z\in I$.
\end{corollary}
\begin{example}
	Let $f\left(x\right)= \log x$,  which is
	a non-negative  ${\rm{A_tA_1}}$-concave for all $0<x<1$.
	Applying Corollary \ref{cor3}, we get
	\begin{align*}
	\left( {x + z} \right)\left( {y + z} \right)\left( {x + y} \right)
	\ge \frac{8}{3}\left( {x + y + z} \right)\left( {xyz} \right),
	\end{align*}
	for all $0< x,y,z< 1$.
\end{example}

\begin{corollary}
	\label{cor4}In Theorem \ref{thm1}.
	\begin{enumerate}
		\item If $h : J \to \left(0,\infty\right)$ is a nonnegative is
		superadditive and $f : I \to \left(0,\infty\right)$ is an
		${\rm{A_tA_h}}$-convex and subadditive, then
		\begin{multline*}
		f\left( {x+y+z} \right)\le f\left( {\frac{{x + z}}{2}}
		\right)+f\left( {\frac{{y + z}}{2}} \right)+f\left( {\frac{{x +
					y}}{2}} \right)
		\\
		\le  h\left( {3/2} \right) f\left( {\frac{{x + y + z}}{3}} \right)
		+h\left(1/2 \right)\left[ {f\left( x \right)+f\left( y
			\right)+f\left( z \right)} \right]
		\\
		\le h\left( {3/2} \right) \left[{f\left( {\frac{x}{3}}
			\right)+f\left( {\frac{y}{3}} \right)+f\left( {\frac{z}{3}}
			\right) }\right] +h\left(1/2 \right)\left[ {f\left( x
			\right)+f\left( y \right)+f\left( z \right)} \right],
		\end{multline*}
		for all $x,y,z\in I$. If $h$ is nonnegative subadditive on $J$ and
		$f$ is an  ${\rm{A_tA_h}}$-concave and superadditive, then the
		inequality is reversed.

		\item If $h : J \to \left(0,\infty\right)$ is a nonnegative is
		superadditive and $f : I \to \left(0,\infty\right)$ is an
		${\rm{A_tA_h}}$-convex  and superadditive, then
		\begin{multline*}
		f\left( {\frac{{x + z}}{2}} \right)+f\left( {\frac{{y + z}}{2}}
		\right)+f\left( {\frac{{x + y}}{2}} \right)
		\\
		\le  h\left( {3/2} \right) f\left( {\frac{{x + y + z}}{3}} \right)
		+h\left(1/2 \right)\left[ {f\left( x \right)+f\left( y
			\right)+f\left( z \right)} \right]
		\\
		\le h\left( {3/2} \right) f\left( {\frac{{x + y + z}}{3}} \right)
		+h\left(1/2 \right)f\left( x+y+z \right),
		\end{multline*}
		for all $x,y,z\in I$. If $h$ is a nonnegative is subadditive and
		$f$ is an ${\rm{A_tA_h}}$-concave and subadditive, then the
		inequality is reversed.
	\end{enumerate}
\end{corollary}

\subsection{The case when $f$ is $h$-${\rm{A_tG_t}}$-convex}

\begin{theorem}
	\label{thm3} Let $h: I\to \left(0,\infty\right)$ be a non-negative
	super(sub)additive
	function. If $f : I \to \left(0,\infty\right)$ be   an
	${\rm{A_tG_h}}$-convex (concave) function, then
	\begin{align}
	f\left( {\frac{{x + z}}{2}} \right)f\left( {\frac{{y + z}}{2}} \right)f\left( {\frac{{x + y}}{2}}
	\right)
	\le \, (\ge)
	\left[ {f\left( {\frac{{x + y + z}}{3}} \right)} \right]^{h\left(
		{3/2} \right)} \left[ {f\left( x \right)f\left( y \right)f\left( z
		\right)} \right]^{h\left( {1/2} \right)},\label{eq2.2}
	\end{align}
	for all $x,y,z\in I$.
\end{theorem}

\begin{proof}
	$f$ is  ${\rm{A_tG_h}}$-convex iff the inequality
	\begin{align*}
	f\left( {t\alpha  + \left( {1 -
			t} \right)\beta } \right) \le \left[ {f\left( \alpha  \right)}
	\right]^{h\left( t \right)} \left[ {f\left( \beta \right)}
	\right]^{h\left(1- t \right)}, \qquad  0\le t\le 1
	\end{align*}
	holds for all $\alpha,\beta \in I$. As in the proof of Theorem
	\ref{thm1}, we have
	$\left(x+y-2z\right)\left(s+t-\frac{3}{2}\right)=0$. If
	$x+y-2z=0$, then $x = y = z$, and Popoviciu's inequality holds.
	
	If $s+t=\frac{3}{2}$, then since $f$ is ${\rm{A_tG_t}}$-convex, we
	have
	\begin{align*}
	f\left( {\frac{{x + z}}{2}} \right) = f\left[ {s\left( {\frac{{x + y + z}}{3}} \right) + \left( {1 - s} \right)z}
	\right]
	\le \left[ {f\left( {\frac{{x + y + z}}{3}} \right)} \right]^{h\left( s \right)} \left[ {f\left( z \right)} \right]^{h\left( {1 - s} \right)}
	\end{align*}
	\begin{align*}
	f\left( {\frac{{y + z}}{2}} \right) = f\left[ {t\left( {\frac{{x + y + z}}{3}} \right) + \left( {1 - t} \right)z}
	\right]
	\le \left[ {f\left( {\frac{{x + y + z}}{3}} \right)} \right]^{h\left( t \right)} \left[ {f\left( z \right)} \right]^{h\left( {1 - t}
		\right)}
	\end{align*}
	and
	\begin{align*}
	f\left( {\frac{{x + y}}{2}} \right) \le \left[ {f\left( x
		\right)f\left( y \right)} \right]^{h\left( {1/2} \right)}
	\end{align*}
	Multiplying these inequalities we get
	\begin{align*}
	&f\left( {\frac{{x + z}}{2}} \right)f\left( {\frac{{y + z}}{2}} \right)f\left( {\frac{{x + y}}{2}} \right) \\
	&\le \left[ {f\left( {\frac{{x + y + z}}{3}} \right)} \right]^{h\left( s \right)} \left[ {f\left( z \right)} \right]^{h\left( {1 - s} \right)} \left[ {f\left( {\frac{{x + y + z}}{3}} \right)} \right]^{h\left( t \right)} \left[ {f\left( z \right)} \right]^{h\left( {1 - t} \right)} \left[ {f\left( x \right)f\left( y \right)} \right]^{h\left( {1/2} \right)}  \\
	&= \left[ {f\left( {\frac{{x + y + z}}{3}} \right)} \right]^{h\left( s \right) + h\left( t \right)} \left[ {f\left( z \right)} \right]^{h\left( {1 - s} \right) + h\left( {1 - t} \right)} \left[ {f\left( x \right)f\left( y \right)} \right]^{h\left( {1/2} \right)}  \\
	&\le \left[ {f\left( {\frac{{x + y + z}}{3}} \right)} \right]^{h\left( {s + t} \right)} \left[ {f\left( z \right)} \right]^{h\left( {2 - s - t} \right)} \left[ {f\left( x \right)f\left( y \right)} \right]^{\frac{1}{2}}  \\
	&= \left[ {f\left( {\frac{{x + y + z}}{3}} \right)} \right]^{h\left( {3/2} \right)} \left[ {f\left( x \right)f\left( y \right)f\left( z \right)} \right]^{h\left( {1/2} \right)}  \\
	\end{align*}
	
\end{proof}

\begin{remark}
Setting $z=y$ 	in \eqref{eq2.2}, then we have
	\begin{align*}
	f^2\left( {\frac{{x + y}}{2}} \right)f\left( {y} \right) \le \,
	(\ge) \left[ {f\left( {\frac{{x + 2y }}{3}} \right)}
	\right]^{h\left( {3/2} \right)} \left[ {f\left( x \right)f^2\left(
		y \right)} \right]^{h\left( {1/2} \right)},
	\end{align*}
	for all $x,y\in I$.
\end{remark}

\begin{corollary}
	\label{cor5}  If $f : I \to \left(0,\infty\right)$ be   an
	${\rm{A_tG_t}}$-convex function, then
	\begin{align*}
	f^2\left( {\frac{{x + z}}{2}} \right)f^2\left( {\frac{{y + z}}{2}}
	\right)f^2\left( {\frac{{x + y}}{2}} \right) \le f^3\left(
	{\frac{{x + y + z}}{3}} \right) f\left( x \right)f\left( y
	\right)f\left( z \right),
	\end{align*}
	for all $x,y,z\in I$. The equality occurred for   $f\left(x\right)={\rm{e}}^x$,
	$x>0$.
\end{corollary}

\begin{example}
	$f\left(x\right)=\cosh \left(x\right)$, $x\in \mathbb{R}$
	is ${\rm{A_tG_t}}$-convex function. Applying Corollary \ref{cor5}
	we get
	\begin{align*}
	{\rm{cosh}}^{\rm{2}} \left( {\frac{{x + z}}{2}}
	\right){\rm{cosh}}^{\rm{2}} \left( {\frac{{y + z}}{2}}
	\right){\rm{cosh}}^{\rm{2}} \left( {\frac{{x + y}}{2}} \right)
	\le {\rm{cosh}}^{\rm{3}} \left( {\frac{{x + y + z}}{3}}
	\right)\cosh \left( x \right)\cosh \left( y \right)\cosh \left( z
	\right)
	\end{align*}
	
\end{example}

\begin{corollary}
	\label{cor6}  If $f : I \to \left(0,\infty\right)$ be   an
	${\rm{A_tG_{1/t}}}$-concave function, then
	\begin{align*}
	f^3\left( {\frac{{x + z}}{2}} \right)f^3\left( {\frac{{y + z}}{2}} \right)f^3\left( {\frac{{x + y}}{2}}
	\right)
	\ge
	f^2\left( {\frac{{x + y + z}}{3}} \right)
	f^6\left( x \right)f^6\left( y \right)f^6\left( z \right),
	\end{align*}
	for all $x,y,z\in I$.
\end{corollary}

\begin{example}
	$f\left(x\right)=\arcsin \left(x\right)$,
	is $\frac{1}{t}$-${\rm{A_tG_t}}$-concave for $x\in[0,1]$. Applying
	Corollary \ref{cor6} we get
	\begin{multline*}
	\arcsin^3\left( {\frac{{x + z}}{2}} \right)\arcsin^3\left(
	{\frac{{y + z}}{2}} \right)\arcsin^3\left( {\frac{{x + y}}{2}}
	\right)
	\\
	\ge \arcsin^2\left( {\frac{{x + y + z}}{3}} \right)
	\arcsin^6\left( x \right)\arcsin^6\left( y \right)\arcsin^6\left(
	z \right),
	\end{multline*}
	for all $0\le x,y,z\le1$.
\end{example}
\begin{corollary}
	\label{cor7} If $f : I \to \left(0,\infty\right)$ be   an
	$1$-${\rm{A_tG_t}}$-concave function, then
	\begin{align*}
	f\left( {\frac{{x + z}}{2}} \right)f\left( {\frac{{y + z}}{2}} \right)f\left( {\frac{{x + y}}{2}}
	\right)
	\le \, (\ge)
	f\left( {\frac{{x + y + z}}{3}} \right) f\left( x \right)f\left( y
	\right)f\left( z \right),
	\end{align*}
	for all $x,y,z\in I$.
\end{corollary}

\begin{example}
	Let  $f\left(x\right)=\arcsin \left(x\right)$, is
	${\rm{A_tG_1}}$-concave for $x\in[0,1]$. Applying Corollary
	\ref{cor7} we get
	\begin{align*}
	\arcsin\left( {\frac{{x + z}}{2}} \right)\arcsin\left( {\frac{{y +
				z}}{2}} \right)\arcsin\left( {\frac{{x + y}}{2}}
	\right)
  \arcsin\left( {\frac{{x + y + z}}{3}} \right) \arcsin\left( x
	\right)\arcsin\left( y \right)\arcsin\left( z \right),
	\end{align*}
	for all $0\le x,y,z \le 1$.
\end{example}
\begin{corollary}
	\label{cor8}In Theorem \ref{thm3}.
	\begin{enumerate}
		\item If $f : I \to \left(0,\infty\right)$ is an
		${\rm{A_tG_h}}$-convex  and submultiplicative,
		\begin{align*}
		f\left(
		{\frac{\left(x+z\right)\left(y+z\right)\left(x+y\right)}{8}}
		\right)&\le f\left( {\frac{{x + z}}{2}} \right)f\left( {\frac{{y +
					z}}{2}} \right)f\left( {\frac{{x + y}}{2}} \right)  
		\\
		&\le \left[ {f\left( {\frac{{x + y + z}}{3}} \right)}
		\right]^{h\left( {3/2} \right)} \left[ {f\left( x \right)f\left( y
			\right)f\left( z \right)} \right]^{h\left( {1/2} \right)},
		\end{align*}
		for all $x,y,z\in I$. If $f$ is an $h$-${\rm{A_tG_t}}$-concave and
		supermultiplicative, then the inequality is reversed.

		\item If $f : I \to \left(0,\infty\right)$ is an
		${\rm{A_tG_h}}$-convex  and supermultiplicative, then
		\begin{align*}
		f\left( {\frac{{x + z}}{2}} \right)f\left( {\frac{{y + z}}{2}}
		\right)f\left( {\frac{{x + y}}{2}} \right)
		&\le \left[ {f\left( {\frac{{x + y + z}}{3}} \right)}
		\right]^{h\left( {3/2} \right)} \left[ {f\left( x \right)f\left( y
			\right)f\left( z \right)} \right]^{h\left( {1/2} \right)}
		\\
		&\le \left[ {f\left( {\frac{{x + y + z}}{3}} \right)}
		\right]^{h\left( {3/2} \right)} \left[ {f\left( xyz \right)}
		\right]^{h\left( {1/2} \right)},
		\end{align*}
		for all $x,y,z\in I$. If $f$ is an ${\rm{A_tA_h}}$-concave and
		submultiplicative, then the inequality is reversed.
	\end{enumerate}
\end{corollary}

\begin{corollary}
	\label{cor9}In Theorem \ref{thm3}.
	\begin{enumerate}
		\item If $f : I \to \left(0,\infty\right)$ is an
		${\rm{A_tG_h}}$-convex  and superadditive,
		\begin{multline*}
		\left[{f\left( {\frac{x }{2}} \right) +f\left( {\frac{z}{2}}
			\right)}\right]\left[{f\left( {\frac{y }{2}} \right) +f\left(
			{\frac{z}{2}} \right)}\right]\left[{f\left( {\frac{x }{2}} \right)
			+f\left( {\frac{y}{2}} \right)}\right]
		\\
		\le f\left( {\frac{{x + z}}{2}} \right)f\left( {\frac{{y + z}}{2}}
		\right)f\left( {\frac{{x + y}}{2}} \right)
		\\
		\le \left[ {f\left( {\frac{{x + y + z}}{3}} \right)}
		\right]^{h\left( {3/2} \right)} \left[ {f\left( x \right)f\left( y
			\right)f\left( z \right)} \right]^{h\left( {1/2} \right)},
		\end{multline*}
		for all $x,y,z\in I$. If $f$ is an  ${\rm{A_tG_h}}$-concave and
		subadditive, then the inequality is reversed.

		\item If $f : I \to \left(0,\infty\right)$ is an
		${\rm{A_tG_h}}$-convex  and subadditive, then
		\begin{multline*}
		f\left( {\frac{{x + z}}{2}} \right)f\left( {\frac{{y + z}}{2}}
		\right)f\left( {\frac{{x + y}}{2}} \right)
		\\
		\le \left[ {f\left( {\frac{{x + y + z}}{3}} \right)}
		\right]^{h\left( {3/2} \right)} \left[ {f\left( x \right)f\left( y
			\right)f\left( z \right)} \right]^{h\left( {1/2} \right)}
		\\
		\le \left[ {f\left( {\frac{{x }}{3}} \right)+f\left( {\frac{{y
				}}{3}} \right)+f\left( {\frac{{ z}}{3}} \right)} \right]^{h\left(
			{3/2} \right)} \left[ {f\left( x \right)f\left( y \right)f\left( z
			\right)} \right]^{h\left( {1/2} \right)},
		\end{multline*}
		for all $x,y,z\in I$. If $f$ is an ${\rm{A_tG_h}}$-concave and
		submultiplicative, then the inequality is reversed.
	\end{enumerate}
\end{corollary}

\subsection{The case when $f$ is  ${\rm{A_tH_h}}$-convex}

\begin{theorem}
	\label{thm4} Let $h: I\to \left(0,\infty\right)$ be a non-negative
	super(sub)additive
	function. If $f : I \to \left(0,\infty\right)$ is an
	${\rm{A_tH_h}}$-concave (convex), then
	\begin{align}
	&\frac{1}{{f\left( {\frac{{x + z}}{2}} \right)}} +
	\frac{1}{{f\left( {\frac{{y + z}}{2}} \right)}} +
	\frac{1}{{f\left( {\frac{{x + y}}{2}} \right)}}\nonumber
	\\
	&\le \,(\ge)\,h\left( {1/2} \right)\left[ {\frac{1}{{f\left( y
				\right)}} + \frac{1}{{f\left( x \right)}} + \frac{1}{{f\left( z
				\right)}}} \right] + \frac{{h\left( {3/2} \right)}}{{f\left(
			{{\textstyle{{x + y + z} \over 3}}} \right)}}, \label{eq2.3}
	\end{align}
	for all $x,y,z\in I$.
\end{theorem}

\begin{proof}
	$f$ is  ${\rm{A_tH_h}}$-convex iff the inequality
	\begin{align*}
	f\left( {t\alpha  + \left( {1 -
			t} \right)\beta } \right) \le\frac{{f\left( \alpha \right)f\left(
			\beta  \right)}}{{h\left( 1-t \right)f\left( \alpha  \right) +
			h\left( {1 - t} \right)f\left( \beta  \right)}}, \qquad  0\le t\le
	1
	\end{align*}
	holds for all $\alpha,\beta \in I$. As in the proof of Theorem
	\ref{thm1}, we have
	$\left(x+y-2z\right)\left(s+t-\frac{3}{2}\right)=0$. If
	$x+y-2z=0$, then $x = y = z$, and Popoviciu's inequality holds.
	
	If $s+t=\frac{3}{2}$, then since $f$ is ${\rm{A_tH_h}}$-convex, we
	have
	\begin{align*}
	f\left( {\frac{{x + z}}{2}} \right) = f\left[ {s\left( {\frac{{x + y + z}}{3}} \right) + \left( {1 - s} \right)z}
	\right]
	\ge
	\frac{{f\left( {{\textstyle{{x + y + z} \over 3}}} \right)f\left(
			z \right)}}{{h\left( {1 - s} \right)f\left( {{\textstyle{{x + y +
							z} \over 3}}} \right) + h\left( s \right)f\left( z \right)}},
	\end{align*}
	and this equivalent to write
	\begin{align}
	\frac{1}{{f\left( {\frac{{x + z}}{2}} \right)}} \le \frac{{h\left(
			{1 - s} \right)f\left( {{\textstyle{{x + y + z} \over 3}}} \right)
			+ h\left( s \right)f\left( z \right)}}{{f\left( {{\textstyle{{x +
							y + z} \over 3}}} \right)f\left( z \right)}}, \label{eq2.4}
	\end{align}
	similarly,
	\begin{align*}
	f\left( {\frac{{y + z}}{2}} \right) = f\left[ {t\left( {\frac{{x + y + z}}{3}} \right) + \left( {1 - t} \right)z}
	\right]
	\ge
	\frac{{f\left( {{\textstyle{{x + y + z} \over 3}}} \right)f\left(
			z \right)}}{{h\left( {1 - t} \right)f\left( {{\textstyle{{x + y +
							z} \over 3}}} \right) + h\left( t \right)f\left( z \right)}},
	\end{align*}
	which equivalent to write
	\begin{align}
	\frac{1}{{f\left( {\frac{{y + z}}{2}} \right)}} \le \frac{{h\left(
			{1 - t} \right)f\left( {{\textstyle{{x + y + z} \over 3}}} \right)
			+ h\left( t \right)f\left( z \right)}}{{f\left( {{\textstyle{{x +
							y + z} \over 3}}} \right)f\left( z \right)}}, \label{eq2.5}
	\end{align}
	and
	\begin{align}
	f\left( {\frac{{x + y}}{2}} \right) &\ge \frac{{f\left( x
			\right)f\left( y \right)}}{{h\left( {1/2} \right)\left( {f\left( x
				\right) + f\left( y \right)} \right)}}\nonumber
	\\
	\Longleftrightarrow\frac{1}{{f\left( {\frac{{x + y}}{2}} \right)}}
	&\le \frac{{h\left( {1/2} \right)\left( {f\left( x \right) +
				f\left( y \right)} \right)}}{{f\left( x \right)f\left( y
			\right)}},  \label{eq2.6}
	\end{align}
	Summing the inequalities \eqref{eq2.4}--\eqref{eq2.6}, we get
	\begin{align*}
	&\frac{1}{{f\left( {\frac{{x + z}}{2}} \right)}} + \frac{1}{{f\left( {\frac{{y + z}}{2}} \right)}} + \frac{1}{{f\left( {\frac{{x + y}}{2}} \right)}} \\
	&\le \frac{{h\left( {1 - s} \right)f\left( {{\textstyle{{x + y + z} \over 3}}} \right) + h\left( s \right)f\left( z \right)}}{{f\left( {{\textstyle{{x + y + z} \over 3}}} \right)f\left( z \right)}} + \frac{{h\left( {1 - t} \right)f\left( {{\textstyle{{x + y + z} \over 3}}} \right) + h\left( t \right)f\left( z \right)}}{{f\left( {{\textstyle{{x + y + z} \over 3}}} \right)f\left( z \right)}} \\&\qquad+ \frac{{h\left( {1/2} \right)\left( {f\left( x \right) + f\left( y \right)} \right)}}{{f\left( x \right)f\left( y \right)}} \\
	&= \frac{{\left[ {h\left( {1 - s} \right) + h\left( {1 - t} \right)} \right]f\left( {{\textstyle{{x + y + z} \over 3}}} \right) + \left[ {h\left( s \right) + h\left( t \right)} \right]f\left( z \right)}}{{f\left( {{\textstyle{{x + y + z} \over 3}}} \right)f\left( z \right)}} \\&\qquad+ \frac{{h\left( {1/2} \right)\left( {f\left( x \right) + f\left( y \right)} \right)}}{{f\left( x \right)f\left( y \right)}} \\
	&\le \frac{{h\left( {2 - s - t} \right)f\left( {{\textstyle{{x + y + z} \over 3}}} \right) + h\left( {s + t} \right)f\left( z \right)}}{{f\left( {{\textstyle{{x + y + z} \over 3}}} \right)f\left( z \right)}} + \frac{{h\left( {1/2} \right)\left( {f\left( x \right) + f\left( y \right)} \right)}}{{f\left( x \right)f\left( y \right)}} \\
	&= \frac{{h\left( {1/2} \right)f\left( {{\textstyle{{x + y + z} \over 3}}} \right) + h\left( {3/2} \right)f\left( z \right)}}{{f\left( {{\textstyle{{x + y + z} \over 3}}} \right)f\left( z \right)}} + \frac{{h\left( {1/2} \right)\left( {f\left( x \right) + f\left( y \right)} \right)}}{{f\left( x \right)f\left( y \right)}} \\
	&=h\left( {1/2} \right)\left[ {\frac{1}{{f\left( y \right)}} +
		\frac{1}{{f\left( x \right)}} + \frac{1}{{f\left( z \right)}}}
	\right] + \frac{{h\left( {3/2} \right)}}{{f\left( {{\textstyle{{x
							+ y + z} \over 3}}} \right)}}
	\end{align*}
	
\end{proof}
\begin{remark}
	In \eqref{eq2.3}, setting $z=y$, then we have
	\begin{align*}
	\frac{2}{{f\left( {\frac{{x + y}}{2}} \right)}} +
	\frac{1}{{f\left( {\frac{{y + z}}{2}} \right)}} \le
	\,(\ge)\,h\left( {1/2} \right)\left[ {\frac{2}{{f\left( y
				\right)}} + \frac{1}{{f\left( x \right)}} } \right] +
	\frac{{h\left( {3/2} \right)}}{{f\left( {{\textstyle{{x + 2y}
						\over 3}}} \right)}},
	\end{align*}
	for all $x,y\in I$.
\end{remark}

\begin{corollary}
	\label{cor10}  If $f : I \to \left(0,\infty\right)$ is an
	${\rm{A_tH_t}}$-concave (convex), then
	\begin{align*}
	\frac{2}{3}\left[{\frac{1}{{f\left( {\frac{{x + z}}{2}} \right)}}
		+ \frac{1}{{f\left( {\frac{{y + z}}{2}} \right)}} +
		\frac{1}{{f\left( {\frac{{x + y}}{2}} \right)}}}\right] \le
	\,(\ge)\,\frac{1}{3}\left[ {\frac{1}{{f\left( y \right)}} +
		\frac{1}{{f\left( x \right)}} + \frac{1}{{f\left( z \right)}}}
	\right] + \frac{{1}}{{f\left( {{\textstyle{{x + y + z} \over 3}}}
			\right)}},
	\end{align*}
	for all $x,y,z\in I$. The equality holds with
	$f\left(x\right)=\frac{1}{x}$, $x>0$.
\end{corollary}

\begin{example}
	Let $f\left(x\right)=x^p$, $p\ge 1$. Then ${\rm{A_tH_t}}$-concave
	for $x\ge  1$. Applying Corollary \ref{cor10}, we get
	\begin{align*}
	\frac{2}{3}\left[ {\left( {\frac{{x + z}}{2}} \right)^{ - p}  +
		\left( {\frac{{y + z}}{2}} \right)^{ - p}  + \left( {\frac{{x +
					y}}{2}} \right)^{ - p} } \right] \le \frac{{x^{ - p}  + y^{ - p} +
			z^{ - p} }}{3} + \left( {\frac{{x + y + z}}{3}} \right)^{ - p}
	\end{align*}
	for all $x,y,z \ge 1$.
\end{example}

\begin{corollary}
	\label{cor11}  If $f : I \to \left(0,\infty\right)$ is an
	${\rm{A_tH_{1/t}}}$-convex, then
	\begin{align*}
	\frac{3}{2}\left[{\frac{1}{{f\left( {\frac{{x + z}}{2}} \right)}}
		+ \frac{1}{{f\left( {\frac{{y + z}}{2}} \right)}} +
		\frac{1}{{f\left( {\frac{{x + y}}{2}} \right)}}}\right] \le
	3\left[ {\frac{1}{{f\left( y \right)}} +
		\frac{1}{{f\left( x \right)}} + \frac{1}{{f\left( z \right)}}}
	\right] + \frac{{1}}{{f\left( {{\textstyle{{x + y + z} \over 3}}}
			\right)}},
	\end{align*}
	for all $x,y,z\in I$.
\end{corollary}
\begin{example}
	Let $f\left(x\right)= - \log \left(x\right)$, $x\gneqq 1$. Then,
	$f$ is  ${\rm{A_tH_{1/t}}}$-convex for  $x\gneqq 1$.
	Applying Corollary \ref{cor11}, we get
	\begin{align*}
	\frac{3}{2}\left[ {\frac{1}{{\log \left( {\frac{{x + z}}{2}}
				\right)}} + \frac{1}{{\log \left( {\frac{{y + z}}{2}} \right)}} +
		\frac{1}{{\log \left( {\frac{{x + y}}{2}} \right)}}} \right] \le
	3\left( {\frac{1}{{\log x}} + \frac{1}{{\log y}} + \frac{1}{{\log
				z}}} \right) + \log \left( {xyz} \right)^{\frac{1}{3}},
	\end{align*}
	for all $x,y,z\gneqq 1$.
\end{example}

\begin{corollary}
\label{cor12}	If $f : I \to \left(0,\infty\right)$ is an
	${\rm{A_tH_1}}$-convex, then
	\begin{align*}
	\frac{1}{{f\left( {\frac{{x + z}}{2}} \right)}} +
	\frac{1}{{f\left( {\frac{{y + z}}{2}} \right)}} +
	\frac{1}{{f\left( {\frac{{x + y}}{2}} \right)}} \le \left[
	{\frac{1}{{f\left( y \right)}} + \frac{1}{{f\left( x \right)}} +
		\frac{1}{{f\left( z \right)}}} \right] + \frac{{1}}{{f\left(
			{{\textstyle{{x + y + z} \over 3}}} \right)}},
	\end{align*}
	for all $x,y,z\in I$.
\end{corollary}
\begin{example}
	Let $f\left(x\right)=-\log \left(x\right)$,  $x\gneqq 1$. Then,
	$f$ is  ${\rm{A_tH_1}}$-convex on  $x\gneqq 1$. Applying
	Corollary \ref{cor12}, we get
	\begin{align*}
	\frac{1}{{\log \left( {\frac{{x + z}}{2}}
			\right)}} + \frac{1}{{\log \left( {\frac{{y + z}}{2}} \right)}} +
	\frac{1}{{\log \left( {\frac{{x + y}}{2}} \right)}} \le
	\frac{1}{{\log x}} + \frac{1}{{\log y}} + \frac{1}{{\log z}} +
	\log \left( {xyz} \right)^{\frac{1}{3}},
	\end{align*}
	for all  $x,y,z\gneqq 1$.
\end{example}

\section{Popoviciu  inequalities for $h$-${\rm{GN}}$-convex
	functions}

\subsection{The case when $f$ is  ${\rm{G_tA_h}}$-convex}

\begin{theorem}
	\label{thm5} Let $h: I\to \left(0,\infty\right)$ be a non-negative
	super(sub)additive
	function. If $f : I \to \left(0,\infty\right)$ is
	${\rm{G_tA_h}}$-convex function, then
	\begin{align}
	f\left( {\sqrt{xz}} \right)+f\left( {\sqrt{yz}} \right)+f\left(
	{\sqrt{xy}} \right) \le \,(\ge)\, h\left( {3/2} \right) f\left(
	{\sqrt[3]{xyz}} \right) +h\left(1/2 \right)\left[ {f\left( x
		\right)+f\left( y \right)+f\left( z \right)} \right],
	\label{eq3.1}
	\end{align}
	for all $x,y,z\in I$.
\end{theorem}

\begin{proof}
	$f$ is  ${\rm{G_tA_h}}$-convex iff the inequality
	\begin{align*}
	f\left( {\alpha ^t \beta ^{1 - t} } \right) \le h\left( {t}
	\right)f\left( \alpha \right) + h\left( {1 - t} \right)f\left(
	\beta  \right), \qquad 0\le t\le 1
	\end{align*}
	holds for all $\alpha,\beta \in I$. Assume that $x\le y \le z$. If
	$y \le \left( {xyz} \right)^{1/3}$, then
	\begin{align*}
	\left( {xyz} \right)^{1/3}  \le \left( {xz} \right)^{1/2}  \le z
	\,\,\text{and}\,\,  \left( {xyz} \right)^{1/3}  \le \left( {yz}
	\right)^{1/2}  \le z,
	\end{align*}
	so that there exist two numbers $s,t \in \left[0,1\right]$
	satisfying
	\begin{align*}
	\left( {xz} \right)^{1/2}  = \left( {xyz} \right)^{s/3} z^{1 - s}
	\end{align*}
	and
	\begin{align*}
	\left( {yz} \right)^{1/2}  = \left( {xyz} \right)^{t/3} z^{1 - t}
	\end{align*}
	Multiplying the above equations, we get $$\left( {xyz}
	\right)^{1/2} z^{1/2}  = \left( {xyz} \right)^{\left( {s + t}
		\right)/3} z^{2 - \left( {s + t} \right)} $$ or $$\left( {xyz}
	\right)^{\frac{{\left( {s + t} \right)}}{3} - \frac{1}{2}} z^{2 -
		\left( {s + t} \right) - \frac{1}{2}}  = 1.$$ If $xyz^2=1$, then
	$x = y = z$, and Popoviciu's inequality holds.
	
	If $s+t=\frac{3}{2}$, then since $f$ is ${\rm{G_tA_h}}$-convex, we
	have
	\begin{align*}
	f\left( {\sqrt {xz} } \right) &= f\left[ {\left( {xyz} \right)^{s/3} z^{1 - s} } \right] \le h\left( s \right)\left[ {f\left( {\sqrt[3]{xyz}  } \right)} \right] + h\left( {1 - s} \right)\left[ {f\left( z \right)} \right] \\
	f\left( {\sqrt {yz} } \right) &= f\left[ {\left( {xyz} \right)^{t/3} z^{1 - t} } \right] \le h\left( t \right)\left[ {f\left( {\sqrt[3]{xyz}} \right)} \right] + h\left( {1 - t} \right)\left[ {f\left( z \right)} \right] \\
	f\left( {\sqrt {xy} } \right) &\le h\left( {\frac{1}{2}} \right)\left[ {f\left( x \right) + f\left( y \right)} \right]
	\end{align*}
	Summing up these inequalities, we get
	\begin{align*}
	&f\left( {\sqrt{xz}} \right)+f\left( {\sqrt{yz}} \right)+f\left( {\sqrt{xy}} \right) \\
	&\le  h\left( s \right)f\left( {\sqrt[3]{xyz}} \right)+h\left( {1 - s} \right)f\left( z
	\right) +h\left( t \right)f\left( {\sqrt[3]{xyz}} \right)+h\left( {1 - t} \right)f\left( z \right)\\
	&\qquad+ h\left( {1/2}\right)\left[ {f\left( x \right)+f\left( y \right)}
	\right]\\
	&= \left[ {h\left( s \right) + h\left( t \right)} \right]f\left( {\sqrt[3]{xyz}} \right)+ \left[ {h\left( {1 - s} \right) + h\left( {1 - t} \right)}  \right]f\left( z \right)+  h\left( {1/2}\right)\left[ {f\left( x \right)+f\left( y \right)}
	\right] \\
	&\le h\left( {s + t} \right) f\left( {\sqrt[3]{xyz}} \right) +h\left( {2 - s - t} \right)  f\left( z \right)+h\left(1/2 \right)\left[ {f\left( x \right)+f\left( y \right)} \right]   \\
	&= h\left( {3/2} \right) f\left( {\sqrt[3]{xyz}} \right) +h\left( {1/2} \right)  f\left( z \right)+h\left(1/2 \right)\left[ {f\left( x \right)+f\left( y \right)} \right] \\
	&= h\left( {3/2} \right) f\left( {\sqrt[3]{xyz}} \right)    +h\left(1/2 \right)\left[ {f\left( x \right)+f\left( y \right)+f\left( z \right)}
	\right],
	\end{align*}
	which proves the inequality \eqref{eq3.1}.
\end{proof}

\begin{remark}
Setting $z=y$ 	in \eqref{eq3.1}, we get
	\begin{align*}
	2f\left( {\sqrt{xy}} \right)+f\left( {y} \right)\le \,(\ge)\,
	h\left( {3/2} \right) f\left( {\sqrt[3]{xy^2}} \right) +h\left(1/2
	\right)\left[ {f\left( x \right)+2f\left( y \right)} \right],
	\end{align*}
	for all $x,y\in I$.
\end{remark}

\begin{corollary}
	\label{cor13} If $f : I \to \left(0,\infty\right)$ is
	${\rm{G_tA_t}}$-convex function, then
	\begin{align*}
	\frac{2}{3}\left[ {f\left( {\sqrt {xz} } \right) + f\left( {\sqrt
			{yz} } \right) + f\left( {\sqrt {xy} } \right)} \right] \le
	f\left( {\sqrt[3]{{xyz}}} \right) + \frac{{f\left( x \right) +
			f\left( y \right) + f\left( z \right)}}{3},
	\end{align*}
	for all $x,y,z\in I$. The equality holds with
	$f\left(x\right)=\log\left(x\right)$, $x>1$.
\end{corollary}
\begin{example}
	Let $f\left(x\right)=\cosh\left(x\right)$, $x>0$. Then, $f$ is
	${\rm{G_tA_t}}$-convex on $(0,\infty)$. Applying Corollary
	\ref{cor13} we get
	\begin{align*}
	\frac{2}{3}\left[ {\cosh \left( {\sqrt {xz} } \right) + \cosh
		\left( {\sqrt {yz} } \right) + \cosh \left( {\sqrt {xy} } \right)}
	\right] \le \cosh \left( {\sqrt[3]{{xyz}}} \right) + \frac{{\cosh
			\left( x \right) + \cosh \left( y \right) + \cosh \left( z
			\right)}}{3},
	\end{align*}
	for all $x,y,z > 0$.
\end{example}

\begin{corollary}
	\label{cor14}If $f : I \to \left(0,\infty\right)$ is
	${\rm{G_tA_{1/t}}}$-concave function, then
	\begin{align*}
	\frac{3}{2}\left[{f\left( {\sqrt{xz}} \right)+f\left( {\sqrt{yz}}
		\right)+f\left( {\sqrt{xy}} \right)}\right] \ge  f\left(
	{\sqrt[3]{xyz}} \right) +3 \left(f\left( x \right)+f\left( y
	\right)+f\left( z \right)\right)
	\end{align*}
	for all $x,y,z\in I$.
\end{corollary}
\begin{example}
	Let $f\left(x\right)=-x^2$, $x>0$. Then, $f$ is
	${\rm{G_tA_{1/t}}}$-concave on $(0,\infty)$. Applying
	Corollary \ref{cor14} we get
	\begin{align*}
	\frac{3}{2}\left({xz + yz+xy } \right) \le  \left(
	{\sqrt[3]{{xyz}}} \right)^2 +3 \left(x^2+y^2+z^2\right)
	\end{align*}
	for all $x,y,z > 0$.
\end{example}

\begin{corollary}
	\label{cor15}If $f : I \to \left(0,\infty\right)$ is
	${\rm{G_tA_1}}$-concave function, then
	\begin{align*}
	f\left( {\sqrt{xz}} \right)+f\left( {\sqrt{yz}} \right)+f\left(
	{\sqrt{xy}} \right)\ge     f\left( {\sqrt[3]{xyz}} \right) +
	f\left( x \right)+f\left( y \right)+f\left( z \right),
	\end{align*}
	for all $x,y,z\in I$.
\end{corollary}

\begin{example}
	Let $f\left(x\right)=-x^2$, $x>0$. Then, $f$ is
	${\rm{G_tA_1}}$-convex on $(0,\infty)$. Applying Corollary
	\ref{cor15} we get
	\begin{align*}
	xz + yz+xy   \le  \left( {\sqrt[3]{{xyz}}} \right)^2 + x^2+y^2+z^2
	\end{align*}
	for all $x,y,z > 0$.
\end{example}

\begin{corollary}
	\label{cor16}In Theorem \ref{thm5}.
	\begin{enumerate}
		\item If $f : I \to \left(0,\infty\right)$ is an
		${\rm{G_tA_h}}$-convex  and superadditive,
		\begin{align*}
		f\left( {\sqrt{xz}} \right)+f\left( {\sqrt{yz}} \right)+f\left(
		{\sqrt{xy}} \right)
		&\le h\left( {3/2} \right) f\left( {\sqrt[3]{xyz}} \right)
		+h\left(1/2 \right)\left[ {f\left( x \right)+f\left( y
			\right)+f\left( z \right)} \right]
		\\
		&\le h\left( {3/2} \right) f\left( {\sqrt[3]{xyz}} \right)
		+h\left(1/2 \right) f\left( x +y+z\right),
		\end{align*}
		for all $x,y,z\in I$. If $f$ is an ${\rm{G_tA_h}}$-concave and
		subadditive, then the inequality is reversed.

		\item If $f : I \to \left(0,\infty\right)$ is an
		${\rm{G_tA_h}}$-convex  and subadditive, then
		\begin{align*}
		f\left( {\sqrt{xz}+\sqrt{yz}+\sqrt{xy}} \right) &\le f\left(
		{\sqrt{xz}} \right)+f\left( {\sqrt{yz}} \right)+f\left(
		{\sqrt{xy}} \right) \\&\le   h\left( {3/2} \right) f\left(
		{\sqrt[3]{xyz}} \right) +h\left(1/2 \right)\left[ {f\left( x
			\right)+f\left( y \right)+f\left( z \right)} \right],
		\end{align*}
		for all $x,y,z\in I$. If $f$ is an  ${\rm{G_tA_h}}$-concave and
		superadditive, then the inequality is reversed.
	\end{enumerate}
\end{corollary}

\begin{example}
	Let $f\left(x\right)=\cosh\left(x\right)$, which is
	${\rm{G_tA_t}}$-convex and superadditive on $(0,\infty)$. Applying
	Corollary \ref{cor16} we get
	\begin{align*}
	\frac{2}{3}\left[ {\cosh \left( {\sqrt {xz} } \right) + \cosh
		\left( {\sqrt {yz} } \right) + \cosh \left( {\sqrt {xy} } \right)}
	\right]
	&\le \cosh \left( {\sqrt[3]{{xyz}}} \right) + \frac{{\cosh \left( x
			\right) + \cosh \left( y \right) + \cosh \left( z \right)}}{3}
	\\
	&\le \cosh\left( {\sqrt[3]{xyz}} \right) + \frac{1}{3}\cosh\left( x
	+y+z\right),
	\end{align*}
	for all $x,y,z > 0$.
\end{example}

\subsection{The case when $f$ is  ${\rm{G_tG_h}}$-convex}

\begin{theorem}
	\label{thm6}Let $h: I\to \left(0,\infty\right)$ be a non-negative
	super(sub)additive
	function. If $f : I \to \left(0,\infty\right)$ is
	${\rm{G_tG_h}}$-convex function, then
	\begin{align}
	f\left( {\sqrt {xz}} \right)f\left( {\sqrt {yz}} \right)f\left(
	{\sqrt {xy}} \right) \le \,(\ge)\,\left[ {f\left( {\sqrt[3]{xyz} }
		\right)} \right]^{h\left( 3/2 \right)}  \left[ {f\left( x
		\right)f\left( y \right)f\left( z \right)} \right]^{h\left( {1/2}
		\right)},\label{eq3.2}
	\end{align}
	for all $x,y,z\in I$.
\end{theorem}

\begin{proof}
	$f$ is  ${\rm{G_tG_h}}$-convex iff the inequality
	\begin{align*}
	f\left( {\alpha ^t \beta ^{1 - t} } \right) \le \left[ {f\left(
		\alpha  \right)} \right]^{h\left( {t} \right)} \left[ {f\left(
		\beta  \right)} \right]^{h\left( {1 - t} \right)}, \qquad  0\le
	t\le 1
	\end{align*}
	holds for all $\alpha, \beta \in I$. As in the proof of Theorem
	\ref{thm4}, if $xyz^2=1$, then $x = y = z$, and Popoviciu's
	inequality holds.
	
	If $s+t=\frac{3}{2}$, then since $f$ is ${\rm{G_tG_h}}$-convex, we
	have
	\begin{align*}
	f\left( {\sqrt {xz} } \right) &= f\left[ {\left( {xyz} \right)^{s/3} z^{1 - s} } \right] \le \left[ {f\left( {\sqrt[3]{xyz} } \right)} \right]^{h\left( s \right)}  \left[ {f\left( z \right)} \right]^{h\left( {1 - s} \right)}, \\
	f\left( {\sqrt {yz} } \right) &= f\left[ {\left( {xyz} \right)^{t/3} z^{1 - t} } \right] \le \left[ {f\left( { \sqrt[3]{xyz}} \right)} \right]^{h\left( t \right)}  \left[ {f\left( z \right)} \right]^{h\left( {1 - t} \right)}, \\
	f\left( {\sqrt {xy} } \right) &\le h\left( {\frac{1}{2}} \right)\left[ {f\left( x \right) + f\left( y \right)}
	\right].
	\end{align*}
	Multiplying these inequalities we get
	\begin{align*}
	&f\left( {\sqrt {xz}} \right)f\left( {\sqrt {yz}} \right)f\left({\sqrt {xy}} \right) \\
	&\le \left[ {f\left( {\sqrt[3]{xyz} } \right)} \right]^{h\left( s \right)}  \left[ {f\left( z \right)} \right]^{h\left( {1 - s} \right)}\left[ {f\left( {\sqrt[3]{xyz} } \right)} \right]^{h\left( t \right)}  \left[ {f\left( z \right)} \right]^{h\left( {1 - t} \right)}  \left[ {f\left( x \right)f\left( y \right)} \right]^{h\left( {1/2} \right)}  \\
	&= \left[ {f\left( {\sqrt[3]{xyz} } \right)} \right]^{h\left( s \right)+h\left( t \right)}  \left[ {f\left( z \right)} \right]^{h\left( {1 - s} \right)+h\left( {1 - t} \right)} \left[ {f\left( x \right)f\left( y \right)} \right]^{h\left( {1/2} \right)}  \\
	&\le \left[ {f\left( {\sqrt[3]{xyz} } \right)} \right]^{h\left( s+t \right)}  \left[ {f\left( z \right)} \right]^{h\left( {2 - s - t} \right)} \left[ {f\left( x \right)f\left( y \right)} \right]^{h\left( {1/2} \right)}  \\
	&= \left[ {f\left( {\sqrt[3]{xyz} } \right)} \right]^{h\left( 3/2 \right)}  \left[ {f\left( z \right)} \right]^{h\left( {1/2} \right)} \left[ {f\left( x \right)f\left( y \right)} \right]^{h\left( {1/2}
		\right)}\\
	&= \left[ {f\left( {\sqrt[3]{xyz} } \right)} \right]^{h\left( 3/2 \right)}  \left[ {f\left( x \right)f\left( y \right)f\left( z \right)} \right]^{h\left( {1/2} \right)},
	\end{align*}
\end{proof}

\begin{remark}
Setting $z=y$ 	in \eqref{eq3.2}, we get
	\begin{align*}
	f^2\left( {\sqrt {xy}} \right)f\left( {y} \right)  \le
	\,(\ge)\,\left[ {f\left( {\sqrt[3]{xy^2} } \right)}
	\right]^{h\left( 3/2 \right)}  \left[ {f\left( x \right)f^2\left(
		y \right)} \right]^{h\left( {1/2} \right)},
	\end{align*}
	for all $x,y\in I$.
\end{remark}
\begin{corollary}
	\label{cor17} If $f : I \to \left(0,\infty\right)$ is
	${\rm{G_tG_t}}$-convex (concave) function, then
	\begin{align*}
	f^2\left( {\sqrt {xz}} \right)f^2\left( {\sqrt {yz}}
	\right)f^2\left( {\sqrt {xy}} \right) \le (\ge)  f^3\left(
	{\sqrt[3]{xyz} } \right) f\left( x \right)f\left( y \right)f\left(
	z \right),
	\end{align*}
	for all $x,y,z\in I$. The equality holds with
	$f\left(x\right)={\rm{e}}^x$, $x>0$.
\end{corollary}
\begin{example}
	Let $f\left(x\right)=\cosh\left(x\right)$, which is
	${\rm{G_tG_t}}$-convex  on $(0,\infty)$. Applying Corollary
	\ref{cor17} we get
	\begin{align*}
	\cosh^2\left( {\sqrt {xz}} \right)\cosh^2\left( {\sqrt {yz}}
	\right)\cosh^2\left( {\sqrt {xy}} \right) \le   f^3\left(
	{\sqrt[3]{xyz} } \right) \cosh\left( x \right)\cosh\left( y
	\right)\cosh\left( z \right),
	\end{align*}
	for all $x,y,z > 0$.
\end{example}

\begin{corollary}
	\label{cor18} If $f : I \to \left(0,\infty\right)$ is
	${\rm{G_tG_1/t}}$-concave function, then
	\begin{align*}
	f^3\left( {\sqrt {xz}} \right)f^3\left( {\sqrt {yz}}
	\right)f^3\left( {\sqrt {xy}} \right)  \ge f^2\left(
	{\sqrt[3]{xyz} } \right)   f^6\left( x \right)f^6\left( y
	\right)f^6\left( z \right),
	\end{align*}
	for all $x,y,z\in I$.
\end{corollary}
\begin{example}
	Let $f\left(x\right)= \exp\left(-x\right)$  which is
	$\frac{1}{t}$-${\rm{G_tG_t}}$-concave  on $(0,\infty)$. Applying
	Corollary \ref{cor18} we get
	\begin{align*}
	\sqrt {xz} + \sqrt {yz}+  \sqrt {xy}   \le
	\frac{2}{3}\sqrt[3]{xyz}+2x+2y+2z,
	\end{align*}
	for all $x,y,z > 0$.
\end{example}

\begin{corollary}
	\label{cor19} If $f : I \to \left(0,\infty\right)$ is
	${\rm{G_tG_1}}$-concave function, then
	\begin{align*}
	f\left( {\sqrt {xz}} \right)f\left( {\sqrt {yz}} \right)f\left(
	{\sqrt {xy}} \right) \le  f\left( {\sqrt[3]{xyz} } \right)f\left(
	x \right)f\left( y \right)f\left( z \right),
	\end{align*}
	for all $x,y,z\in I$.
\end{corollary}

\begin{example}
	Let $f\left(x\right)=\exp\left(-x\right)$, which is
	${\rm{G_tG_1}}$-concave  on $(0,\infty)$. Applying Corollary
	\ref{cor19} we get
	\begin{align*}
	\sqrt {xz}+ \sqrt {yz}+\sqrt {xy} \le  \sqrt[3]{xyz} + x+y+z,
	\end{align*}
	for all $x,y,z > 0$.
\end{example}

\begin{corollary}
	\label{cor20}In Theorem \ref{thm6}.
	\begin{enumerate}
		\item If $f : I \to \left(0,\infty\right)$ is an
		${\rm{G_tG_h}}$-convex  and supermultiplicative,
		\begin{align*}
		f\left( {\sqrt {xz}} \right)f\left( {\sqrt {yz}} \right)f\left(
		{\sqrt {xy}} \right) &\le  \left[ {f\left( {\sqrt[3]{xyz} }
			\right)} \right]^{h\left( 3/2 \right)}  \left[ {f\left( x
			\right)f\left( y \right)f\left( z \right)} \right]^{h\left( {1/2}
			\right)}
		\\
		&\le  \left[ {f\left( {\sqrt[3]{xyz} } \right)} \right]^{h\left(
			3/2 \right)}  \left[ {f\left( xyz \right)} \right]^{h\left( {1/2}
			\right)},
		\end{align*}
		for all $x,y,z\in I$.
		
		\item If $f : I \to \left(0,\infty\right)$ is an
		${\rm{G_tG_h}}$-convex  and submultiplicative, then
		\begin{align*}
		f\left( {xzy} \right) &\le f\left( {\sqrt {xz}} \right)f\left(
		{\sqrt {yz}} \right)f\left( {\sqrt {xy}} \right)
		\\
		&\le  \left[ {f\left( {\sqrt[3]{xyz} } \right)} \right]^{h\left(
			3/2 \right)}  \left[ {f\left( x \right)f\left( y \right)f\left( z
			\right)} \right]^{h\left( {1/2} \right)}
		\\
		&\le  \left[ {f\left( {\sqrt[3]{x} } \right)f\left( {\sqrt[3]{y} }
			\right)f\left( {\sqrt[3]{z} } \right)} \right]^{h\left( 3/2
			\right)}  \left[ {f\left( x \right)f\left( y \right)f\left( z
			\right)} \right]^{h\left( {1/2} \right)},
		\end{align*}
		for all $x,y,z\in I$.
	\end{enumerate}
\end{corollary}

\begin{example}
	Let $f\left(x\right)=\cosh\left(x\right)$, which is
	${\rm{G_tG_t}}$-convex and supermultiplicative on $[1,\infty)$.
	Applying Corollary \ref{cor20} we get
	\begin{align*}
	\cosh^2\left( {\sqrt {xz}} \right)\cosh^2\left( {\sqrt {yz}}
	\right)\cosh^2\left( {\sqrt {xy}} \right) &\le   \cosh^3\left(
	{\sqrt[3]{xyz} } \right) \cosh\left( x \right)\cosh\left( y
	\right)\cosh\left( z \right)
	\\
	&\le   \cosh^3\left( {\sqrt[3]{xyz} } \right)   \cosh\left( xyz
	\right)
	\end{align*}
	for all $x,y,z \ge 1$.
\end{example}

\subsection{The case when $f$ is  ${\rm{G_tH_h}}$-convex}

\begin{theorem}
\label{thm7}	Let $h: I\to \left(0,\infty\right)$ be a non-negative
	super(sub)additive
	function. If $f : I \to \left(0,\infty\right)$  is
	${\rm{G_tH_h}}$-concave (convex) function, then
	\begin{align}
	&\frac{1}{{f\left( {\sqrt {xz} } \right)}} + \frac{1}{{f\left(
			{\sqrt {yz} } \right)}} + \frac{1}{{f\left( {\sqrt {xy} }
			\right)}}\nonumber
	\\
	&\le \,(\ge)\,h\left( {\frac{1}{2}} \right)\left[
	{\frac{1}{{f\left( x \right)}} + \frac{1}{{f\left( y \right)}} +
		\frac{1}{{f\left( z \right)}}} \right] + \frac{{h\left( {3/2}
			\right)}}{{f\left( {\sqrt[3]{{xyz}}} \right)}}, \label{eq3.3}
	\end{align}
	for all $x,y,z\in I$.
\end{theorem}

\begin{proof}
	$f$ is  ${\rm{G_tH_h}}$-convex iff the inequality
	\begin{align*}
	f\left( {\alpha ^t \beta ^{1 - t} } \right) \le \frac{{f\left(
			\alpha  \right)f\left( \beta  \right)}}{{h\left( {1- t}
			\right)f\left( \alpha  \right) + h\left( { t} \right)f\left( \beta
			\right)}}, \qquad  0\le t\le 1
	\end{align*}
	holds for all $\alpha, \beta \in I$. As in the proof of Theorem \ref{thm4}, if
	$xyz^2=1$, then $x = y = z$, and Popoviciu's inequality holds.
	
	If $s+t=\frac{3}{2}$, then since $f$ is ${\rm{G_tH_h}}$-convex, we
	have
	\begin{align*}
	f\left( {\sqrt {xz} } \right) = f\left[ {\left( {xyz}
		\right)^{s/3} z^{1 - s} } \right] \ge \frac{{f\left(
			{\sqrt[3]{{xyz}}} \right)f\left( z \right)}}{{h\left( {1 - s}
			\right)f\left( {\sqrt[3]{{xyz}}} \right) + h\left( s
			\right)f\left( z \right)}}
	\end{align*}
	and this equivalent to write
	\begin{align}
	\frac{1}{{f\left( {\sqrt {xz} } \right)}} \le \frac{{h\left( {1 -
				s} \right)f\left( {\sqrt[3]{{xyz}}} \right) + h\left( s
			\right)f\left( z \right)}}{{f\left( {\sqrt[3]{{xyz}}}
			\right)f\left( z \right)}},\label{eq3.4}
	\end{align}
	similarly,
	\begin{align*}
	f\left( {\sqrt {yz} } \right) = f\left[ {\left( {xyz}
		\right)^{t/3} z^{1 - t} } \right] \ge \frac{{f\left(
			{\sqrt[3]{{xyz}}} \right)f\left( z \right)}}{{h\left( {1 - t}
			\right)f\left( {\sqrt[3]{{xyz}}} \right) + h\left( t
			\right)f\left( z \right)}}
	\end{align*}
	which equivalent to write
	\begin{align}
	\frac{1}{{f\left( {\sqrt {yz} } \right)}} \le \frac{{h\left( {1 -
				t} \right)f\left( {\sqrt[3]{{xyz}}} \right) + h\left( t
			\right)f\left( z \right)}}{{f\left( {\sqrt[3]{{xyz}}}
			\right)f\left( z \right)}},\label{eq3.5}
	\end{align}
	and
	\begin{align}
	f\left( {\sqrt {xy} } \right) \ge \frac{{f\left( x \right)f\left(
			y \right)}}{{h\left( {1/2} \right)\left( {f\left( x \right) +
				f\left( y \right)} \right)}}\nonumber
	\\
	\Longleftrightarrow \frac{1}{{f\left( {\sqrt {xy} } \right)}} \le
	\frac{{h\left( {1/2} \right)\left( {f\left( x \right) + f\left( y
				\right)} \right)}}{{f\left( x \right)f\left( y
			\right)}}.\label{eq3.6}
	\end{align}
	Summing the inequalities \eqref{eq3.4}--\eqref{eq3.6}, we get
	\begin{align*}
	&\frac{1}{{f\left( {\sqrt {xz} } \right)}} + \frac{1}{{f\left( {\sqrt {yz} } \right)}} + \frac{1}{{f\left( {\sqrt {xy} } \right)}} \\
	&\le \frac{{h\left( {1 - s} \right)f\left( {\sqrt[3]{{xyz}}} \right) + h\left( s \right)f\left( z \right)}}{{f\left( {\sqrt[3]{{xyz}}} \right)f\left( z \right)}} + \frac{{h\left( {1 - t} \right)f\left( {\sqrt[3]{{xyz}}} \right) + h\left( t \right)f\left( z \right)}}{{f\left( {\sqrt[3]{{xyz}}} \right)f\left( z \right)}} \\&\qquad+ \frac{{h\left( {1/2} \right)\left( {f\left( x \right) + f\left( y \right)} \right)}}{{f\left( x \right)f\left( y \right)}} \\
	&= \frac{{\left[ {h\left( {1 - s} \right) + h\left( {1 - t} \right)} \right]f\left( {\sqrt[3]{{xyz}}} \right) + \left[ {h\left( s \right) + h\left( t \right)} \right]f\left( z \right)}}{{f\left( {\sqrt[3]{{xyz}}} \right)f\left( z \right)}} \\&\qquad+ \frac{{h\left( {1/2} \right)\left( {f\left( x \right) + f\left( y \right)} \right)}}{{f\left( x \right)f\left( y \right)}} \\
	&\le \frac{{h\left( {2 - s - t} \right)f\left( {\sqrt[3]{{xyz}}} \right) + h\left( {s + t} \right)f\left( z \right)}}{{f\left( {\sqrt[3]{{xyz}}} \right)f\left( z \right)}} + \frac{{h\left( {1/2} \right)\left( {f\left( x \right) + f\left( y \right)} \right)}}{{f\left( x \right)f\left( y \right)}} \\
	&\le \frac{{h\left( {1/2} \right)f\left( {\sqrt[3]{{xyz}}} \right) + h\left( {3/2} \right)f\left( z \right)}}{{f\left( {\sqrt[3]{{xyz}}} \right)f\left( z \right)}} + \frac{{h\left( {1/2} \right)\left( {f\left( x \right) + f\left( y \right)} \right)}}{{f\left( x \right)f\left( y \right)}} \\
	&= h\left( {\frac{1}{2}} \right)\left[ {\frac{1}{{f\left( x \right)}} + \frac{1}{{f\left( y \right)}} + \frac{1}{{f\left( z \right)}}} \right] + \frac{{h\left( {3/2} \right)}}{{f\left( {\sqrt[3]{{xyz}}}
			\right)}},
	\end{align*}
	which proves the inequality in \eqref{eq3.3}.
\end{proof}

\begin{remark}
Setting $z=y$  in \eqref{eq3.3}, then we get
	\begin{align*}
	\frac{2}{{f\left( {\sqrt {xy} } \right)}} + \frac{1}{{f\left( {y}
			\right)}} \le \,(\ge)\,h\left( {\frac{1}{2}} \right)\left[
	{\frac{1}{{f\left( x \right)}} + \frac{2}{{f\left( y \right)}}  }
	\right] + \frac{{h\left( {3/2} \right)}}{{f\left(
			{\sqrt[3]{{xy^2}}} \right)}},
	\end{align*}
	for all $x,y\in I$.
\end{remark}

\begin{corollary}
	\label{cor21}If $f : I \to \left(0,\infty\right)$  is
	${\rm{G_tH_t}}$-concave (convex) function, then
	\begin{align*}
	\frac{2}{3}\left[{\frac{1}{{f\left( {\sqrt {xz} } \right)}} +
		\frac{1}{{f\left( {\sqrt {yz} } \right)}} + \frac{1}{{f\left(
				{\sqrt {xy} } \right)}}}\right]\le \,(\ge)\,\frac{1}{3}\left[
	{\frac{1}{{f\left( x \right)}} + \frac{1}{{f\left( y \right)}} +
		\frac{1}{{f\left( z \right)}}} \right] +  \frac{{1}}{{f\left(
			{\sqrt[3]{{xyz}}} \right)}},
	\end{align*}
	for all $x,y,z\in I$. The equality holds with
	$f\left(x\right)=\frac{1}{\log\left(x\right)}$, $x \gvertneqq 1$.
\end{corollary}

\begin{example}
	Let $f\left(x\right)=\cosh \left(x\right)$, then $f$ is
	${\rm{G_tH_t}}$-convex for all $x \ge1$. Applying Corollary
	\ref{cor21}, then we get
	\begin{multline*}
	\frac{2}{3}\left[{\frac{1}{{\cosh\left( {\sqrt {xz} } \right)}} +
		\frac{1}{{\cosh\left( {\sqrt {yz} } \right)}} +
		\frac{1}{{\cosh\left( {\sqrt {xy} } \right)}}}\right]
	\\
	\ge \frac{1}{3}\left[ {\frac{1}{{\cosh\left( x \right)}} +
		\frac{1}{{\cosh\left( y \right)}} + \frac{1}{{\cosh\left( z
				\right)}}} \right] +  \frac{{1}}{{\cosh\left( {\sqrt[3]{{xyz}}}
			\right)}},
	\end{multline*}
	for all $x,y,z \ge1$.
\end{example}

\begin{corollary}
	\label{cor22} If $f : I \to \left(0,\infty\right)$  is
	${\rm{G_tH_{1/t}}}$-convex function, then
	\begin{align*}
	\frac{3}{2}\left[{\frac{1}{{f\left( {\sqrt {xz} } \right)}} +
		\frac{1}{{f\left( {\sqrt {yz} } \right)}} + \frac{1}{{f\left(
				{\sqrt {xy} } \right)}}}\right]\ge  3\left[ {\frac{1}{{f\left( x
				\right)}} + \frac{1}{{f\left( y \right)}} + \frac{1}{{f\left( z
				\right)}}} \right] +  \frac{{1}}{{f\left( {\sqrt[3]{{xyz}}}
			\right)}},
	\end{align*}
	for all $x,y,z\in I$.
\end{corollary}

\begin{example}
	Let $f\left(x\right)=-\log\left(x\right)$, then $f$ is
	${\rm{G_tH_{1/t}}}$-convex for all $x>1$. Applying
	Corollary \ref{cor22}, then we get
	\begin{align*}
	\frac{3}{2}\left[{\frac{1}{{\log\left( {\sqrt {xz} } \right)}} +
		\frac{1}{{\log\left( {\sqrt {yz} } \right)}} +
		\frac{1}{{\log\left( {\sqrt {xy} } \right)}}}\right]\le
	3\left[ {\frac{1}{{\log\left( x \right)}} +
		\frac{1}{{\log\left( y \right)}} + \frac{1}{{\log\left( z
				\right)}}} \right] +  \frac{{1}}{{\log\left( {\sqrt[3]{{xyz}}}
			\right)}},
	\end{align*}
	for all $x,y,z > 1$.
\end{example}

\begin{corollary}
	\label{cor23} If $f : I \to \left(0,\infty\right)$  is
	$1$-${\rm{G_tH_t}}$-convex function, then
	\begin{align*}
	\frac{1}{{f\left( {\sqrt {xz} } \right)}} + \frac{1}{{f\left(
			{\sqrt {yz} } \right)}} + \frac{1}{{f\left( {\sqrt {xy} }
			\right)}} \ge  \left[ {\frac{1}{{f\left( x \right)}} +
		\frac{1}{{f\left( y \right)}} + \frac{1}{{f\left( z \right)}}}
	\right] + \frac{1}{{f\left( {\sqrt[3]{{xyz}}} \right)}},
	\end{align*}
	for all $x,y,z\in I$.
\end{corollary}

\begin{example}
	Let $f\left(x\right)=-\log\left(x\right)$, then $f$ is
	${\rm{G_tH_1}}$-convex for all $x>1$. Applying Corollary
	\ref{cor23}, then we get
	\begin{align*}
	\frac{1}{{\log\left( {\sqrt {xz} } \right)}} +
	\frac{1}{{\log\left( {\sqrt {yz} } \right)}} +
	\frac{1}{{\log\left( {\sqrt {xy} } \right)}} \le  \left[
	{\frac{1}{{\log\left( x \right)}} + \frac{1}{{\log\left( y
				\right)}} + \frac{1}{{\log\left( z \right)}}} \right] +
	\frac{1}{{\log\left( {\sqrt[3]{{xyz}}} \right)}},
	\end{align*}
	for all $x,y,z >1$.
\end{example}

\section{Popoviciu  inequalities for $h$-${\rm{HN}}$-convex
	functions}

\subsection{The case when $f$ is  ${\rm{H_tA_h}}$-convex}

\begin{theorem}
	\label{thm8} Let $h: I\to \left(0,\infty\right)$ be a non-negative
	super(sub)additive.
	If $f : I \to \left(0,\infty\right)$  is
	${\rm{H_tA_h}}$-convex (concave) function, then
	\begin{align}
	&f\left( {\frac{{2xz}}{{x + z}}} \right) + f\left(
	{\frac{{2yz}}{{y + z}}} \right) + f\left( {\frac{{2xy}}{{x + y}}}
	\right)\nonumber
	\\
	&\le \,(\ge)\,h\left( {3/2} \right)f\left( {\frac{{3xyz}}{{xy + yz
				+ xz}}} \right) + h\left( {1/2} \right)\left[ {f\left( x \right) +
		f\left( y \right) + f\left( z \right)} \right], \label{eq4.1}
	\end{align}
	for all $x,y,z\in I$.
\end{theorem}

\begin{proof}
	$f$ is ${\rm{H_tA_h}}$-convex iff the inequality
	\begin{align*}
	f\left( {\frac{{\alpha \beta }}{{t\alpha  + \left( {1 - t}
				\right)\beta }}} \right) \le h\left( {1-t} \right)f\left( \alpha
	\right) + h\left( { t} \right)f\left( \beta  \right), \qquad 0\le
	t\le 1,
	\end{align*}
	holds for all $\alpha,\beta \in I$. Assume that $x\le y \le z$. If
	$ y \le \frac{{3xyz}}{{xy + yz + xz}}$, then
	\begin{align*}
	\frac{{3xyz}}{{xy + yz + xz}} \le \frac{{2xz}}{{x + z}} \le z
	\,\,\text{and}\,\,  \frac{{3xyz}}{{xy + yz + xz}} \le
	\frac{{2yz}}{{y + z}} \le z,
	\end{align*}
	so that there exist two numbers $s,t \in \left[0,1\right]$
	satisfying
	\begin{align*}
	\frac{{2xz}}{{x + z}} = \frac{{{\textstyle{{3xyz} \over {xy + yz +
						xz}}} \cdot z}}{{s{\textstyle{{3xyz} \over {xy + yz + xz}}} +
			\left( {1 - s} \right)z}},
	\end{align*}
	and
	\begin{align*}
	\frac{{2yz}}{{y + z}} = \frac{{{\textstyle{{3xyz} \over {xy + yz +
						xz}}} \cdot z}}{{t{\textstyle{{3xyz} \over {xy + yz + xz}}} +
			\left( {1 - t} \right)z}}.
	\end{align*}
	For simplicity set, $u=\frac{{3xyz}}{{xy + yz + xz}}$, summing the
	reciprocal of the previous two equations
	\begin{align*}
	\frac{{x + z}}{{2xz}} + \frac{{y + z}}{{2yz}} = \frac{{\left( {s +
				t} \right){\textstyle{{3xyz} \over {xy + yz + xz}}} + \left( {2 -
				s - t} \right)z}}{{{\textstyle{{3xyz} \over {xy + yz + xz}}} \cdot
			z}} = \frac{{3\left( {s + t} \right)u + \left( {2 - s - t}
			\right)z}}{{3u \cdot z}}.
	\end{align*}
	Simplifying the above equation and reverse it back to the original
	form (taking the reciprocal again), we get
	\begin{align*}
	\frac{u}{{u + z}} = \frac{u}{{2\left( {s + t} \right)u +
			\frac{2}{3}\left( {2 - s - t} \right)z}},
	\end{align*}
	since
	$y,x,z>0$, this yields that $x=y=z $
	and thus Popoviciu's inequality holds, or $s+t=\frac{1}{2}$ and in
	this case since $f$ is ${\rm{H_tA_h}}$-convex, we have
	\begin{align*}
	f\left( {\frac{{2xz}}{{x + z}}} \right) &= f\left( {\frac{{{\textstyle{{3xyz} \over {xy + yz + xz}}} \cdot z}}{{s{\textstyle{{3xyz} \over {xy + yz + xz}}} + \left( {1 - s} \right)z}}} \right) \le h\left( s \right)f\left( z \right) + h\left( {1 - s} \right)f\left( {\frac{{3xyz}}{{xy + yz + xz}}} \right), \\
	f\left( {\frac{{2yz}}{{y + z}}} \right) &= f\left( {\frac{{{\textstyle{{3xyz} \over {xy + yz + xz}}} \cdot z}}{{t{\textstyle{{3xyz} \over {xy + yz + xz}}} + \left( {1 - t} \right)z}}} \right) \le h\left( t \right)f\left( z \right) + h\left( {1 - t} \right)f\left( {\frac{{3xyz}}{{xy + yz + xz}}} \right), \\
	f\left( {\frac{{2xy}}{{x + y}}} \right) &\le h\left( {1/2} \right)\left[ {f\left( x \right) + f\left( y\right)}
	\right].
	\end{align*}
	Summing up these inequalities we get
	\begin{align*}
	&f\left( {\frac{{2xz}}{{x + z}}} \right) + f\left( {\frac{{2yz}}{{y + z}}} \right) + f\left( {\frac{{2xy}}{{x + y}}} \right) \\
	&\le \left[ {h\left( s \right) + h\left( t \right)} \right]f\left( z \right) + \left[ {h\left( {1 - s} \right) + h\left( {1 - t} \right)} \right]f\left( {\frac{{3xyz}}{{xy + yz + xz}}} \right) + h\left( {1/2} \right)\left[ {f\left( x \right) + f\left( y \right)} \right] \\
	&\le h\left( {s + t} \right)f\left( z \right) + h\left( {2 - s - t} \right)f\left( {\frac{{3xyz}}{{xy + yz + xz}}} \right) + h\left( {1/2} \right)\left[ {f\left( x \right) + f\left( y \right)} \right] \\
	&= h\left( {3/2} \right)f\left( {\frac{{3xyz}}{{xy + yz + xz}}} \right) + h\left( {1/2} \right)\left[ {f\left( x \right) + f\left( y \right) + f\left( z \right)}
	\right],
	\end{align*}
	which proves the inequality in \eqref{eq4.1}.
\end{proof}

\begin{remark}
Setting $z=y$ 	in \eqref{eq4.1}, then we get
	\begin{align*}
	2f\left( {\frac{{2xy}}{{x + y}}} \right) + f\left( {y} \right) \le
	\,(\ge)\,h\left( {3/2} \right)f\left( {\frac{{3xy}}{{2x + y}}}
	\right) + h\left( {1/2} \right)\left[ {f\left( x \right) +
		2f\left( y \right)} \right],
	\end{align*}
	for all $x,y\in I$.
\end{remark}

\begin{corollary}
	\label{cor24} If $f : I \to \left(0,\infty\right)$  is
	${\rm{H_tA_t}}$-convex (concave) function, then
	\begin{align*}
	\frac{2}{3}\left[{f\left( {\frac{{2xz}}{{x + z}}} \right) +
		f\left( {\frac{{2yz}}{{y + z}}} \right) + f\left( {\frac{{2xy}}{{x
					+ y}}} \right)}\right]
	\le \,(\ge)\, f\left( {\frac{{3xyz}}{{xy + yz + xz}}} \right) +
	\frac{f\left( x \right) + f\left( y \right) + f\left( z
		\right)}{3},
	\end{align*}
	for all $x,y,z\in I$. The equality holds with
	$f\left(x\right)=\frac{1}{x}$, $x> 0$.
\end{corollary}

\begin{example}
	Let $f\left(x\right)=\arctan\left(x\right)$, then  $f$ is
	${\rm{H_tA_t}}$-convex on $\left(0,\infty\right)$. Applying
	Corollary \ref{cor24}, then we get
	\begin{multline*}
	\frac{2}{3}\left[{\arctan\left( {\frac{{2xz}}{{x + z}}} \right) +
		\arctan\left( {\frac{{2yz}}{{y + z}}} \right) + \arctan\left(
		{\frac{{2xy}}{{x + y}}} \right)}\right]
	\\
	\le \arctan\left( {\frac{{3xyz}}{{xy + yz + xz}}} \right) +
	\frac{\arctan\left( x \right) + \arctan\left( y \right) +
		\arctan\left( z \right)}{3},
	\end{multline*}
\end{example}

\begin{corollary}
	\label{cor25}If $f : I \to \left(0,\infty\right)$  is
	${\rm{H_tA_{1/t}}}$-concave function, then
	\begin{align*}
	\frac{3}{2}\left[{f\left( {\frac{{2xz}}{{x + z}}} \right) +
		f\left( {\frac{{2yz}}{{y + z}}} \right) + f\left( {\frac{{2xy}}{{x
					+ y}}} \right)}\right]
	\ge f\left( {\frac{{3xyz}}{{xy + yz + xz}}} \right) + 3\left[
	{f\left( x \right) + f\left( y \right) + f\left( z \right)}
	\right],
	\end{align*}
	for all $x,y,z\in I$.
\end{corollary}

\begin{example}
	Let $f\left(x\right)=x^2$, therefore $f$ is
	${\rm{H_tA_{1/t}}}$-concave on $x<0$. Applying Corollary
	\ref{cor25}, then we get
	\begin{align*}
	\left( {\frac{{xz}}{{x + z}}} \right)^2  +  \left( {\frac{{yz}}{{y
				+ z}}} \right)^2 +\left( {\frac{{xy}}{{x + y}}} \right)^2 \ge
	\frac{3}{2}\left( {\frac{{ xyz}}{{xy + yz + xz}}} \right)^2 +
	\frac{1}{18}\left( {x^2+y^2+z^2 } \right),
	\end{align*}
	for all $x,y,z<0$.
\end{example}
\begin{corollary}
	\label{cor26}If $f : I \to \left(0,\infty\right)$  is
	${\rm{H_tA_1}}$-concave function, then
	\begin{align*}
	&f\left( {\frac{{2xz}}{{x + z}}} \right) + f\left(
	{\frac{{2yz}}{{y + z}}} \right) + f\left( {\frac{{2xy}}{{x + y}}}
	\right)\nonumber
	\\
	&\ge f\left( {\frac{{3xyz}}{{xy + yz + xz}}} \right) +
	\left[ {f\left( x \right) + f\left( y \right)
		+ f\left( z \right)} \right],
	\end{align*}
	for all $x,y,z\in I$.
\end{corollary}

\begin{example}
	Let $f\left(x\right)=x^2$, therefore  $f$ is
	${\rm{H_tA_1}}$-concave on $\left(-\infty,0\right)$. Applying
	Corollary \ref{cor26}, then we get
	\begin{align*}
	\left( {\frac{{xz}}{{x + z}}} \right)^2 + \left( {\frac{{yz}}{{y
				+ z}}} \right)^2 + \left( {\frac{{xy}}{{x + y}}} \right)^2 \ge
	\frac{9}{4}\left[{\frac{x^2+y^2+z^2}{9}+\left( {\frac{{xyz}}{{xy +
					yz + xz}}} \right)^2 } \right],
	\end{align*}
	for all $x,y,z<0$.
\end{example}

\begin{corollary}
	\label{cor27}In Theorem \ref{thm8}.
	\begin{enumerate}
		\item If $f : I \to \left(0,\infty\right)$ is an
		${\rm{H_tA_h}}$-convex and superadditive, then
		\begin{align*}
		&2\left[{f\left( {\frac{{xz}}{{x + z}} } \right)+f\left( {
				\frac{{yz}}{{y + z}} } \right)+f\left( { \frac{{xy}}{{x + y}}}
			\right)}\right]
		\\
		& \le f\left( {\frac{{2xz}}{{x + z}}} \right) + f\left(
		{\frac{{2yz}}{{y + z}}} \right) + f\left( {\frac{{2xy}}{{x + y}}}
		\right)\nonumber
		\\
		&\le  h\left( {3/2} \right)f\left( {\frac{{3xyz}}{{xy + yz + xz}}}
		\right) + h\left( {1/2} \right)\left[ {f\left( x \right) + f\left(
			y \right) + f\left( z \right)} \right]
		\\
		&\le  h\left( {3/2} \right)f\left( {\frac{{3xyz}}{{xy + yz + xz}}}
		\right) + h\left( {1/2} \right)f\left( x+y+z \right),
		\end{align*}
		for all $x,y,z\in I$. If $f$ is an  ${\rm{H_tA_h}}$-concave and
		subadditive, then the inequality is reversed.

		\item If $f : I \to \left(0,\infty\right)$ is an
		${\rm{H_tA_h}}$-convex  and subadditive, then
		\begin{align*}
		&f\left( {\frac{{2xz}}{{x + z}}+\frac{{2yz}}{{y +
					z}}+\frac{{2xy}}{{x + y}}} \right)
		\\
		&\le f\left( {\frac{{2xz}}{{x + z}}} \right) + f\left(
		{\frac{{2yz}}{{y + z}}} \right) + f\left( {\frac{{2xy}}{{x + y}}}
		\right)
		\\
		&\le h\left( {3/2} \right)f\left( {\frac{{3xyz}}{{xy + yz + xz}}}
		\right) + h\left( {1/2} \right)\left[ {f\left( x \right) + f\left(
			y \right) + f\left( z \right)} \right]
		\\
		&\le 3h\left( {3/2} \right)f\left( {\frac{{xyz}}{{xy + yz + xz}}}
		\right)   + h\left( {1/2} \right)\left[ {f\left( x \right) +
			f\left( y \right) + f\left( z \right)} \right],
		\end{align*}
		for all $x,y,z\in I$. If $f$ is an  ${\rm{H_tA_h}}$-concave and
		superadditive, then the inequality is reversed.
	\end{enumerate}
\end{corollary}

\subsection{The case when $f$ is  ${\rm{H_tG_h}}$-convex}

\begin{theorem}
	\label{thm9} Let $h: I\to \left(0,\infty\right)$ be a non-negative
	super(sub)additive.
	If $f : I\to \left(0,\infty\right)$ is
	${\rm{H_tG_h}}$-convex (concave) function, then
	\begin{align}
	&f\left( {\frac{{2xz}}{{x + z}}} \right)  f\left( {\frac{{2yz}}{{y
				+ z}}} \right)   f\left( {\frac{{2xy}}{{x + y}}} \right)\nonumber
	\\
	&\le \,(\ge)\,\left[ {f\left( {\frac{{3xyz}}{{xy + yz + xz}}}
		\right)} \right]^{h\left( {3/2} \right)} \left[ {f\left( x
		\right)f\left( y \right)f\left( z \right)} \right]^{h\left( {1/2}
		\right)}, \label{eq4.2}
	\end{align}
	for all $x,y,z\in I$.
\end{theorem}

\begin{proof}
	$f$ is ${\rm{H_tG_h}}$-convex iff the inequality
	\begin{align*}
	f\left( {\frac{{\alpha \beta }}{{t\alpha  + \left( {1 - t}
				\right)\beta }}} \right) \le \left[ {f\left( \alpha  \right)}
	\right]^{h\left( { 1-t} \right)} \left[ {f\left( \beta  \right)}
	\right]^{h\left( {  t} \right)}, \qquad 0\le t\le 1.
	\end{align*}
	holds for all $\alpha,\beta \in I$. As in the proof of Theorem
	\ref{thm8}, if $x=y=z$, then the inequality holds. If
	$s+t=\frac{1}{2}$ since $f$ is ${\rm{H_tG_h}}$-convex, we have
	\begin{align*}
	f\left( {\frac{{2xz}}{{x + z}}} \right) &= f\left( {\frac{{{\textstyle{{3xyz} \over {xy + yz + xz}}} \cdot z}}{{s{\textstyle{{3xyz} \over {xy + yz + xz}}} + \left( {1 - s} \right)z}}} \right) \le \left[ {f\left( z \right)} \right]^{h\left( s \right)} \left[ {f\left( {\frac{{3xyz}}{{xy + yz + xz}}} \right)} \right]^{h\left( {1 - s} \right)},  \\
	f\left( {\frac{{2yz}}{{y + z}}} \right) &= f\left( {\frac{{{\textstyle{{3xyz} \over {xy + yz + xz}}} \cdot z}}{{t{\textstyle{{3xyz} \over {xy + yz + xz}}} + \left( {1 - t} \right)z}}} \right) \le \left[ {f\left( z \right)} \right]^{h\left( t \right)} \left[ {f\left( {\frac{{3xyz}}{{xy + yz + xz}}} \right)} \right]^{h\left( {1 - t} \right)},  \\
	f\left( {\frac{{2xy}}{{x + y}}} \right) &\le \left[ {f\left( x \right)f\left( y \right)} \right]^{h\left( {1/2}
		\right)}.
	\end{align*}
	Multiplying these inequalities we get
	\begin{align*}
	&f\left( {\frac{{2xz}}{{x + z}}} \right)f\left( {\frac{{2yz}}{{y + z}}} \right)f\left( {\frac{{2xy}}{{x + y}}} \right) \\
	&\le \left[ {f\left( z \right)} \right]^{h\left( s \right)} \left[ {f\left( {\frac{{3xyz}}{{xy + yz + xz}}} \right)} \right]^{h\left( {1 - s} \right)} \left[ {f\left( z \right)} \right]^{h\left( t \right)} \left[ {f\left( {\frac{{3xyz}}{{xy + yz + xz}}} \right)} \right]^{h\left( {1 - t} \right)} \left[ {f\left( x \right)f\left( y \right)} \right]^{h\left( {1/2} \right)}  \\
	&\le \left[ {f\left( z \right)} \right]^{h\left( s \right) + h\left( t \right)} \left[ {f\left( {\frac{{3xyz}}{{xy + yz + xz}}} \right)} \right]^{h\left( {1 - s} \right) + h\left( {1 - t} \right)} \left[ {f\left( x \right)f\left( y \right)} \right]^{h\left( {1/2} \right)}  \\
	&\le \left[ {f\left( z \right)} \right]^{h\left( {s + t} \right)} \left[ {f\left( {\frac{{3xyz}}{{xy + yz + xz}}} \right)} \right]^{h\left( {2 - s - t} \right)} \left[ {f\left( x \right)f\left( y \right)} \right]^{h\left( {1/2} \right)}  \\
	&= \left[ {f\left( z \right)} \right]^{h\left( {1/2} \right)} \left[ {f\left( {\frac{{3xyz}}{{xy + yz + xz}}} \right)} \right]^{h\left( {3/2} \right)} \left[ {f\left( x \right)f\left( y \right)} \right]^{h\left( {1/2} \right)}  \\
	&= \left[ {f\left( {\frac{{3xyz}}{{xy + yz + xz}}} \right)} \right]^{h\left( {3/2} \right)} \left[ {f\left( x \right)f\left( y \right)f\left( z \right)} \right]^{h\left( {1/2}
		\right)},
	\end{align*}
	which proves the inequality in \eqref{eq4.2}.
\end{proof}

\begin{remark}
Setting $z=y$ 	in \eqref{eq4.2}, we get that
	\begin{align*}
	2f\left( {\frac{{2xy}}{{x + y}}} \right)   f\left( {y} \right)\le
	\,(\ge)\,\left[ {f\left( {\frac{{3xy}}{{2x + y}}} \right)}
	\right]^{h\left( {3/2} \right)} \left[ {f\left( x \right)f^2\left(
		y \right)} \right]^{h\left( {1/2} \right)},
	\end{align*}
	for all $x,y\in I$.
\end{remark}

\begin{corollary}
	\label{cor28}If $f : I \to \left(0,\infty\right)$ is
	${\rm{H_tG_t}}$-convex (concave) function, then
	\begin{align*}
	&f\left( {\frac{{2xz}}{{x + z}}} \right)   f\left(
	{\frac{{2yz}}{{y + z}}} \right)   f\left( {\frac{{2xy}}{{x + y}}}
	\right)\nonumber
	\\
	&\le \,(\ge)\,\left[ {f\left( {\frac{{3xyz}}{{xy + yz + xz}}}
		\right)} \right]^{3/2} \left[ {f\left( x \right)f\left( y
		\right)f\left( z \right)} \right]^{1/2},
	\end{align*}
	for all $x,y,z\in I$. The equality holds with
	$f\left(x\right)={\rm{e}}^{\frac{1}{x}}$, $x>0$.
\end{corollary}

\begin{example}
	Let $f\left(x\right)=\exp\left(x\right)$, $x>0$. Then, $f$ is
	${\rm{H_tG_t}}$-convex on $\left(0,\infty\right)$. Applying
	Corollary \ref{cor28} we get
	\begin{align*}
	\frac{{4xz}}{{x + z}}+\frac{{4yz}}{{y + z}}+\frac{{4xy}}{{x +
			y}}\le \frac{{9xyz}}{{xy + yz + xz}}+ xyz,
	\end{align*}
	for all $x,y,z>0$.
\end{example}

\begin{corollary}
	\label{cor29}If $f : I \to \left(0,\infty\right)$ is
	${\rm{H_tG_{1/t}}}$-concave, then
	\begin{align*}
	&f\left( {\frac{{2xz}}{{x + z}}} \right)   f\left(
	{\frac{{2yz}}{{y + z}}} \right)   f\left( {\frac{{2xy}}{{x + y}}}
	\right)\nonumber
	\\
	&\ge\left[ {f\left( {\frac{{3xyz}}{{xy + yz + xz}}} \right)}
	\right]^{2/3} \left[ {f\left( x \right)f\left( y \right)f\left( z
		\right)} \right]^{2},
	\end{align*}
	for all $x,y,z\in I$.
\end{corollary}

\begin{example}
	Let $f\left(x\right)=\exp\left(-x\right)$, $x>0$. Then, $f$ is
	${\rm{H_tG_{1/t}}}$-concave on $\left(0,\infty\right)$.
	Applying Corollary \ref{cor29} we get
	\begin{align*}
	\frac{{ xz}}{{x + z}}+\frac{{ yz}}{{y + z}}+
	\frac{{ xy}}{{x + y}}
	\le \frac{{ xyz}}{{xy + yz + xz}}+  xyz,
	\end{align*}
	for all $x,y,z>0$.
\end{example}
\begin{corollary}
	\label{cor30} If $f : I \to \left(0,\infty\right)$ is
	${\rm{H_tG_1}}$-concave function, then
	\begin{align*}
	&f\left( {\frac{{2xz}}{{x + z}}} \right)   f\left(
	{\frac{{2yz}}{{y + z}}} \right)   f\left( {\frac{{2xy}}{{x + y}}}
	\right)\nonumber
	\\
	&\ge f\left( {\frac{{3xyz}}{{xy + yz + xz}}} \right) f\left( x
	\right)f\left( y \right)f\left( z \right) ,
	\end{align*}
	for all $x,y,z\in I$.
\end{corollary}

\begin{example}
	Let $f\left(x\right)=\exp\left(-x\right)$, $x>0$. Then, $f$ is
	${\rm{H_tG_1}}$-concave on $\left(0,\infty\right)$. Applying
	Corollary \ref{cor30} we get
	\begin{align*}
	\frac{{2xz}}{{x + z}}+\frac{{2yz}}{{y + z}}+\frac{{2xy}}{{x + y}}
	\le  \frac{{3xyz}}{{xy + yz + xz}} + x+y +z
	\end{align*}
	for all $x,y,z >0$.
\end{example}

\begin{corollary}
	\label{cor13}In Theorem \ref{thm9}.
	\begin{enumerate}
		\item If $f : I \to \left(0,\infty\right)$ is an
		${\rm{H_tG_h}}$-convex and superadditive, then
		\begin{align*}
		&2 \left[{f\left( {\frac{{xz}}{{x + z}} } \right)+f\left( {
				\frac{{yz}}{{y + z}} } \right)+f\left( { \frac{{xy}}{{x + y}}}
			\right)}\right]
		\\
		&\le f\left( {\frac{{2xz}}{{x + z}}} \right) + f\left(
		{\frac{{2yz}}{{y + z}}} \right) + f\left( {\frac{{2xy}}{{x + y}}}
		\right)
		\\
		&\le  \left[ {f\left( {\frac{{3xyz}}{{xy + yz + xz}}} \right)}
		\right]^{h\left( {3/2} \right)} \left[ {f\left( x \right)f\left( y
			\right)f\left( z \right)} \right]^{h\left( {1/2} \right)},
		\end{align*}
		for all $x,y,z\in I$. If $f$ is an $h$-${\rm{H_tG_t}}$-concave and
		subadditive, then the inequality is reversed.

		\item If $f : I \to \left(0,\infty\right)$ is an
		${\rm{H_tG_h}}$-convex  and subadditive, then
		\begin{align*}
		&f\left( {\frac{{2xz}}{{x + z}}+\frac{{2yz}}{{y +
					z}}+\frac{{2xy}}{{x + y}}} \right)
		\\
		&\le f\left( {\frac{{2xz}}{{x + z}}} \right) + f\left(
		{\frac{{2yz}}{{y + z}}} \right) + f\left( {\frac{{2xy}}{{x + y}}}
		\right)
		\\
		&\le  \left[ {f\left( {\frac{{3xyz}}{{xy + yz + xz}}} \right)}
		\right]^{h\left( {3/2} \right)} \left[ {f\left( x \right)f\left( y
			\right)f\left( z \right)} \right]^{h\left( {1/2} \right)}
		\\
		&\le  \left[ {3f\left( {\frac{{xyz}}{{xy + yz + xz}}} \right)}
		\right]^{h\left( {3/2} \right)} \left[ {f\left( x \right)f\left( y
			\right)f\left( z \right)} \right]^{h\left( {1/2} \right)},
		\end{align*}
		for all $x,y,z\in I$. If $f$ is an  ${\rm{H_tG_h}}$-concave and
		superadditive, then the inequality is reversed.
	\end{enumerate}
\end{corollary}

\subsection{The case when $f$ is  ${\rm{H_tH_h}}$-convex}

\begin{theorem}
	Let $h: I\to \left(0,\infty\right)$ be a non-negative
	super(sub)additive.
	If $f : I \to \left(0,\infty\right)$ is
	${\rm{H_tH_h}}$-concave (convex) function, then
	\begin{align}
	&f\left( {\frac{{2xz}}{{x + z}}} \right) + f\left(
	{\frac{{2yz}}{{y + z}}} \right) + f\left( {\frac{{2xy}}{{x + y}}}
	\right)\nonumber
	\\
	&\le \,(\ge)\,h\left( {\frac{1}{2}} \right)\left[
	{\frac{1}{{f\left( x \right)}} + \frac{1}{{f\left( y \right)}} +
		\frac{1}{{f\left( z \right)}}} \right] + \frac{{h\left( {3/2}
			\right)}}{{f\left( {{\textstyle{{3xyz} \over {xy + yz + xz}}}}
			\right)}},\label{eq4.3}
	\end{align}
	for all $x,y,z\in I$.
\end{theorem}
\begin{proof}
	$f$ is ${\rm{H_tH_h}}$-convex iff the inequality
	\begin{align*}
	f\left( {\frac{{\alpha \beta }}{{t\alpha  + \left( {1 - t}
				\right)\beta }}} \right) \le \frac{{f\left( \alpha  \right)f\left(
			\beta  \right)}}{{h\left( {t} \right)f\left( \alpha  \right) +
			h\left( {1 - t} \right)f\left( \beta  \right)}}, \qquad  0\le t\le
	1
	\end{align*}
	holds for all $\alpha, \beta \in I$. As in the proof of Theorem
	\ref{thm8}, if $x=y=z$, then the inequality holds. If
	$s+t=\frac{1}{2}$ since $f$ is ${\rm{H_tH_h}}$-convex, we have
	
	\begin{align*}
	f\left( {\frac{{2xz}}{{x + z}}} \right) &= f\left( {\frac{{{\textstyle{{3xyz} \over {xy + yz + xz}}} \cdot z}}{{s{\textstyle{{3xyz} \over {xy + yz + xz}}} + \left( {1 - s} \right)z}}} \right) \ge \frac{{f\left( {{\textstyle{{3xyz} \over {xy + yz + xz}}}} \right) \cdot f\left( z \right)}}{{h\left( s \right)f\left( {{\textstyle{{3xyz} \over {xy + yz + xz}}}} \right) + h\left( {1 - s} \right)f\left( z \right)}}, \\
	f\left( {\frac{{2yz}}{{y + z}}} \right) &= f\left( {\frac{{{\textstyle{{3xyz} \over {xy + yz + xz}}} \cdot z}}{{t{\textstyle{{3xyz} \over {xy + yz + xz}}} + \left( {1 - t} \right)z}}} \right) \ge \frac{{f\left( {{\textstyle{{3xyz} \over {xy + yz + xz}}}} \right) \cdot f\left( z \right)}}{{h\left( t \right)f\left( {{\textstyle{{3xyz} \over {xy + yz + xz}}}} \right) + h\left( {1 - t} \right)f\left( z \right)}}, \\
	f\left( {\frac{{2xy}}{{x + y}}} \right) &\ge \frac{{f\left( x \right)f\left( y \right)}}{{h\left( {1/2} \right)\left[ {f\left( x \right) + f\left( y \right)}
			\right]}},
	\end{align*}
	Therefore, by summing the reciprocal of the above inequalities
	we get
	\begin{align*}
	&\frac{1}{{f\left( {{\textstyle{{2xz} \over {x + z}}}} \right)}} + \frac{1}{{f\left( {{\textstyle{{2yz} \over {y + z}}}} \right)}} + \frac{1}{{f\left( {{\textstyle{{2xy} \over {x + y}}}} \right)}} \\
	&\le \frac{{h\left( s \right)f\left( {{\textstyle{{3xyz} \over {xy + yz + xz}}}} \right) + h\left( {1 - s} \right)f\left( z \right) + h\left( t \right)f\left( {{\textstyle{{3xyz} \over {xy + yz + xz}}}} \right) + h\left( {1 - t} \right)f\left( z \right)}}{{f\left( {{\textstyle{{3xyz} \over {xy + yz + xz}}}} \right) \cdot f\left( z \right)}} \\
	&\qquad+ \frac{{h\left( {1/2} \right)\left[ {f\left( x \right) + f\left( y \right)} \right]}}{{f\left( x \right)f\left( y \right)}} \\
	&\le \frac{{\left[ {h\left( s \right) + h\left( s \right)} \right]f\left( {{\textstyle{{3xyz} \over {xy + yz + xz}}}} \right) + \left[ {h\left( {1 - s} \right) + h\left( {1 - t} \right)} \right]f\left( z \right)}}{{f\left( {{\textstyle{{3xyz} \over {xy + yz + xz}}}} \right) \cdot f\left( z \right)}} + \frac{{h\left( {1/2} \right)\left[ {f\left( x \right) + f\left( y \right)} \right]}}{{f\left( x \right)f\left( y \right)}} \\
	&\le \frac{{h\left( {s + t} \right)f\left( {{\textstyle{{3xyz} \over {xy + yz + xz}}}} \right) + h\left( {2 - s - t} \right)f\left( z \right)}}{{f\left( {{\textstyle{{3xyz} \over {xy + yz + xz}}}} \right) \cdot f\left( z \right)}} + \frac{{h\left( {1/2} \right)\left[ {f\left( x \right) + f\left( y \right)} \right]}}{{f\left( x \right)f\left( y \right)}} \\
	&= \frac{{h\left( {1/2} \right)f\left( {{\textstyle{{3xyz} \over {xy + yz + xz}}}} \right) + h\left( {3/2} \right)f\left( z \right)}}{{f\left( {{\textstyle{{3xyz} \over {xy + yz + xz}}}} \right) \cdot f\left( z \right)}} + \frac{{h\left( {1/2} \right)\left[ {f\left( x \right) + f\left( y \right)} \right]}}{{f\left( x \right)f\left( y \right)}} \\
	&= \frac{{h\left( {1/2} \right)f\left( {{\textstyle{{3xyz} \over {xy + yz + xz}}}} \right) + h\left( {3/2} \right)f\left( z \right)}}{{f\left( {{\textstyle{{3xyz} \over {xy + yz + xz}}}} \right) \cdot f\left( z \right)}} + \frac{{h\left( {1/2} \right)\left[ {f\left( x \right) + f\left( y \right)} \right]}}{{f\left( x \right)f\left( y \right)}} \\
	&= h\left( {\frac{1}{2}} \right)\left[ {\frac{1}{{f\left( x \right)}} + \frac{1}{{f\left( y \right)}} + \frac{1}{{f\left( z \right)}}} \right] + \frac{{h\left( {3/2} \right)}}{{f\left( {{\textstyle{{3xyz} \over {xy + yz + xz}}}}
			\right)}},
	\end{align*}
	which proves the inequality in \eqref{eq4.3}.
\end{proof}

\begin{remark}
Setting $z=y$ 	in \eqref{eq4.3}, then we get
	\begin{align*}
	2f\left( {\frac{{2xy}}{{x + y}}} \right) + f\left( {y} \right) \le
	\,(\ge)\,h\left( {\frac{1}{2}} \right)\left[ {\frac{1}{{f\left( x
				\right)}} + \frac{2}{{f\left( y \right)}} } \right] +
	\frac{{h\left( {3/2} \right)}}{{f\left( {{\textstyle{{3xy} \over
						{2x + y}}}} \right)}},
	\end{align*}
	for all $x,y,z\in I$.
\end{remark}

\begin{corollary}
	\label{cor31}If $f : I \to \left(0,\infty\right)$ is
	${\rm{H_tH_t}}$-concave (convex) function, then
	\begin{align*}
	&\frac{2}{3}\left[ {\frac{1}{f\left( {\frac{{2xz}}{{x + z}}}
			\right)} + \frac{1}{f\left( {\frac{{2yz}}{{y + z}}} \right)} +
		\frac{1}{f\left( {\frac{{2xy}}{{x + y}}} \right)}}\right]\nonumber
	\\
	&\le \,(\ge)\, \frac{1}{3}\left[ {\frac{1}{{f\left( x \right)}} +
		\frac{1}{{f\left( y \right)}} + \frac{1}{{f\left( z \right)}}}
	\right] + \frac{{1}}{{f\left( {{\textstyle{{3xyz} \over {xy + yz +
							xz}}}} \right)}},
	\end{align*}
	for all $x,y,z\in I$. The equality holds with $f\left(x\right)=x$,
	$x>1$.
\end{corollary}
\begin{example}
	Let $f\left(x\right)=\arctan\left(x\right)$, $x>0$. Then $f$ is
	${\rm{H_tH_t}}$-concave  on $\left(0,\infty\right)$. Applying
	Corollary  \ref{cor31}, then we get
	\begin{align*}
	&\frac{2}{3}\left[ {\frac{1}{\arctan\left( {\frac{{2xz}}{{x + z}}}
			\right)} + \frac{1}{\arctan\left( {\frac{{2yz}}{{y + z}}} \right)}
		+ \frac{1}{\arctan\left( {\frac{{2xy}}{{x + y}}} \right)}}\right]
	\\
	&\le  \frac{1}{3}\left[ {\frac{1}{{\arctan\left( x \right)}} +
		\frac{1}{{\arctan\left( y \right)}} + \frac{1}{{\arctan\left( z
				\right)}}} \right] + \frac{{1}}{{\arctan\left( {{\textstyle{{3xyz}
						\over {xy + yz + xz}}}} \right)}},
	\end{align*}
	for all $x,y,z>0$.
\end{example}

\begin{corollary}
	\label{cor32} If $f : I \to \left(0,\infty\right)$ is
	${\rm{H_tH_{1/t}}}$-convex function, then
	\begin{align*}
	&\frac{3}{2}\left[ {\frac{1}{f\left( {\frac{{2xz}}{{x + z}}}
			\right)} + \frac{1}{f\left( {\frac{{2yz}}{{y + z}}} \right)} +
		\frac{1}{f\left( {\frac{{2xy}}{{x + y}}} \right)}}\right]\nonumber
	\\
	&\ge 3\left[ {\frac{1}{{f\left( x \right)}} + \frac{1}{{f\left( y
				\right)}} + \frac{1}{{f\left( z \right)}}} \right] +
	\frac{{1}}{{f\left( {{\textstyle{{3xyz} \over {xy + yz + xz}}}}
			\right)}},
	\end{align*}
	for all $x,y,z\in I$.
\end{corollary}
\begin{example}
	Let $f\left(x\right)=-\log\left(x\right)$, $x>1$. Then $f$ is
	${\rm{H_tH_{1/t}}}$-convex  on $\left(0,\infty\right)$.
	Applying Corollary  \ref{cor32}, then we get
	\begin{align*}
	&\frac{3}{2}\left[ {\frac{1}{\log\left( {\frac{{2xz}}{{x + z}}}
			\right)} + \frac{1}{\log\left( {\frac{{2yz}}{{y + z}}} \right)} +
		\frac{1}{\log\left( {\frac{{2xy}}{{x + y}}} \right)}}\right]
	\\
	&\le  3\left[ {\frac{1}{{\log\left( x \right)}} +
		\frac{1}{{\log\left( y \right)}} + \frac{1}{{\log\left( z
				\right)}}} \right] + \frac{{1}}{{\log\left( {{\textstyle{{3xyz}
						\over {xy + yz + xz}}}} \right)}},
	\end{align*}
	for all $x,y,z>0$.
\end{example}

\begin{corollary}
	\label{cor33}If $f : I \to \left(0,\infty\right)$ is
	${\rm{H_tH_1}}$-convex function, then
	\begin{align*}
	&\frac{1}{f\left( {\frac{{2xz}}{{x + z}}} \right)} +
	\frac{1}{f\left( {\frac{{2yz}}{{y + z}}} \right)} +
	\frac{1}{f\left( {\frac{{2xy}}{{x + y}}} \right)}
	\\
	&\ge\left[ {\frac{1}{{f\left( x \right)}} + \frac{1}{{f\left( y
				\right)}} + \frac{1}{{f\left( z \right)}}} \right] +
	\frac{1}{f\left( {{\textstyle{{3xyz} \over {xy + yz + xz}}}}
		\right)},
	\end{align*}
	for all $x,y,z\in I$.
\end{corollary}
\begin{example}
	Let $f\left(x\right)=-\log\left(x\right)$, $x>0$. Then $f$ is
	${\rm{H_tH_1}}$-convex  on $\left(0,\infty\right)$. Applying
	Corollary  \ref{cor33}, then we get
	\begin{align*}
	&\frac{1}{\log\left( {\frac{{2xz}}{{x + z}}} \right)} +
	\frac{1}{\log\left( {\frac{{2yz}}{{y + z}}} \right)} +
	\frac{1}{\log\left( {\frac{{2xy}}{{x + y}}} \right)}
	\\
	&\le\left[ {\frac{1}{{\log\left( x \right)}} +
		\frac{1}{{\log\left( y \right)}} + \frac{1}{{\log\left( z
				\right)}}} \right] + \frac{1}{\log\left( {{\textstyle{{3xyz} \over
					{xy + yz + xz}}}} \right)},
	\end{align*}
	for all $x,y,z>0$.
\end{example}

\end{document}